\documentclass[10pt]{amsart}
\usepackage[english]{babel}
\usepackage{amsmath}
\usepackage{amsfonts}
\usepackage{amssymb}
\usepackage{bm}
\usepackage{amsthm}

\usepackage[english]{babel}
\usepackage{yfonts}
\usepackage[T1]{fontenc}
\usepackage[utf8x]{inputenc}
\usepackage{enumerate}
\usepackage{verbatim}
\usepackage{graphicx}
\usepackage{verbatim}
\usepackage{faktor}
\usepackage{xcolor}
\usepackage{xfrac}
\usepackage{tikz}
\usepackage[all]{xy}
\usepackage{mathdots}
\usepackage{cancel}

\newcommand{\calA}{\mathcal A}

\newcommand{\calW}{\mathcal W}

\newcommand{\calF}{\mathcal F}

\newcommand{\PSL}{\mathsf{PSL}}
\renewcommand{\r}{\rho}

\renewcommand{\P}{{\rm P}}

\newcommand{\ov}{\overline}

\newcommand{\s}{\sigma}

\newcommand{\di}{\partial_{\infty}}
\newcommand{\Gr}{{\rm Gr}}

\newcommand{\g}{\gamma}

\renewcommand{\s}{\sigma}

\newcommand{\G}{\Gamma}

\newcommand{\Id}{{\rm Id}}
\newcommand{\R}{\mathbb R}
\newcommand{\Z}{\mathbb Z}
\renewcommand{\P}{\mathbb P}

\newcommand{\C}{\mathbb C}
\newcommand{\N}{\mathbb N}

\newcommand{\SL}{\mathsf{SL}}
\newcommand{\PGL}{\mathsf{PGL}}
\newcommand{\PO}{\mathsf{PO}}

\newcommand{\GL}{\mathsf{GL}}

\renewcommand{\>}{\rangle}
\newcommand{\bord}{\partial}
\newcommand{\Hom}{{\rm Hom}}

\newcommand{\un}{\underline}

\newcommand{\dual}{\flat}

\newcommand{\sZ}{\mathsf{Z}}
\newcommand{\sX}{\mathsf{X}}

\newcommand{\sfp}{\mathsf{p}}

\newcommand{\dg}{\partial_{\infty}\Gamma}

\newcommand{\cro}{{\rm cr}}
\newcommand{\gcr}{{\rm cr}}
\newcommand{\pcr}{{\rm pcr}}
\newcommand{\tv}{\hspace{1mm}\pitchfork\hspace{1mm}}
\newcommand{\ntv}{\hspace{1mm}\cancel{\pitchfork}\hspace{1mm}}
\newcommand{\bx}{\mathrm{x} }
\newcommand{\bbx}{s}%{{\color{blue}\mathrm{x}} }
\newcommand{\by}{\mathrm{y} }
\newcommand{\bby}{t}%{{\color{blue}\mathrm{y}} }

\newcommand{\spa}{{\rm span}}

   \newcommand\quotient[2]{
	\mathchoice
	{% \displaystyle
		\text{\raise.5ex\hbox{$#1$}\big/\lower.5ex\hbox{$#2$}}%
	}
	{% \textstyle
		\text{\raise.25ex\hbox{$#1$}\big/\lower.25ex\hbox{$#2$}}%
	}
	{% \scriptstyle
		#1\,/\,#2
	}
	{% \scriptscriptstyle  
		#1\,/\,#2
	}
}

\newcommand{\bpm}{\begin{pmatrix}}
\newcommand{\epm}{\end{pmatrix}}

\theoremstyle{plain}
\newtheorem{thm}{Theorem}[section]
\newtheorem{lem}[thm]{Lemma}
\newtheorem{prop}[thm]{Proposition}
\newtheorem{cor}[thm]{Corollary}

\newtheorem*{teo*}{Theorem}

\newtheorem{question}[thm]{Question}

\newtheorem*{notation}{Notation}
\newtheorem{thmA}{Theorem}

\theoremstyle{definition}
\newtheorem{example}[thm]{Example}

\newtheorem{defn}[thm]{Definition}
\newtheorem{remark}[thm]{Remark}

\newcommand{\thismonth}{\ifcase\month % case 0 --- impossible!
  \or January\or February\or March\or April\or May\or June%
  \or July\or August\or September\or October\or November%
  \or December\fi}

\title[$k$--positive surface group representations]{Degenerations of $k$--positive surface group representations}
\author{Jonas Beyrer and Beatrice Pozzetti}
\date{\today}
\begin{document}
\thanks{J.B. acknowledges funding by the Deutsche Forschungsgemeinschaft (DFG, German Research Foundation) grant 338644254 (SPP2026), and the Schweizerischer Nationalfonds (SNF, Swiss Research Foundation), P2ZHP2 184022 (Early Postdoc.Mobility). B.P acknowledges funding by the DFG, grant 427903332 (Emmy Noether), and is partially supported by the DFG under Germany's Excellence Strategy EXC-2181/1-390900948. Both authors acknowledge funding by the DFG grant 281869850 (RTG 2229). J.B. thanks I.H.E.S. for their hospitality.}
\begin{abstract}
We introduce \emph{$k$--positive representations}, a large class of $\{1,\ldots,k\}$--Anosov surface group representations into $\PGL(E)$ that share many features with Hitchin representations, %but, unless they are Hitchin,  can be deformed into non-discrete representations, 
and we study their degenerations: unless they are Hitchin, they can be deformed to non-discrete representations, but any limit is at least $(k-3)$--positive and irreducible limits are  $(k-1)$--positive.  
%are at least $(k-3)$--positive. 
%After studying  general properties of $k$--positive representations, and their similarities to Hitchin representations, we focus on degenerations. We show that those representations lose the $k$--positivity in a very controlled way. 
A major ingredient, of independent interest, is a general limit theorem for positively ratioed representations.

\end{abstract}

\maketitle
\tableofcontents

\section{Introduction}
{
For a vector space $E$ of dimension $d$, the Hitchin component is an exceptional connected component of the character variety $\Hom(\G,\PGL(E))/\!\!/\PGL(E)$ of the fundamental group $\Gamma$ of a closed hyperbolic surface in the projective linear group of $E$, which was  discovered by Hitchin with the aid of Higgs bundles \cite{Hitchin}. Thanks to the work of Labourie \cite{Labourie-IM}, we know that this component forms an  \emph{higher rank Teichm\"uller space}: it only consists of equivalence classes of injective homomorphism with discrete image.  This was independently proven by Fock-Goncharov  who also found a beautiful link between representations in these components and Lustzig's notion of positivity for $n$-tuples in the (full) flag manifold $\calF(E)$ \cite{FG}.
Many more properties of holonomies of hyperbolizations also hold for Hitchin representations: they satisfy a collar Lemma \cite{LZ}, they have $C^1$ limit sets \cite{PSW1,ZZ}, they admit a Liouville current that can be used to compute their associated length functions via intersection \cite{MZ}. It was later discovered that all these interesting geometric properties are not peculiar to higher rank Teichm\"uller spaces, but hold for more general classes of representations, for example representations satisfying transversality properties $C_k$ and $H_k$ \cite{Beyrer-Pozzetti}, which include small deformations of many Fuchsian loci.

In this paper we  investigate representations that have features in common with Hitchin representations, but are, in general, not contained in a higher rank Teichm\"uller space. The three major contributions of the paper are:  first we single out a class of representations that is rigid enough to have most of the good properties of Hitchin representation (all the above mentioned, together with additional hyperconvexity and positivity properties), but is rich enough to contain a wealth of examples: \emph{$k$-positive representations}. Second, we study new phenomena with respect to higher rank Teichm\"uller spaces: \emph{degenerations}. For $k$--positive representations degenerations can be described precisely. Third we study  positively ratioed representations, a class introduced in \cite{MZ} that includes, but is not limited to, $k$-positive representations, and we prove that their {{proximal}} limits admit continuous equivariant boundary maps  - the latter will also be an important tool in the companion paper \cite{BP3}, which studies some aspects of the higher rank Teichm\"uller theory of $\PO(p,q)$.

A $k$-positive representation is, by definition, $\{1,\ldots, k\}$--Anosov (see Section \ref{s.Anosov} for the definition and references), and thus it admits, for every $l\leq k$ and $l\geq d-k$, an injective equivariant \emph{boundary map} $\xi^l:\bord\G\to\Gr_l(E)$ parametrizing the associated limit set. For the ease of notation we use the shortcut $x^l:=\xi^l(x)\in\Gr_l(E)$. The definition of $k$--positivity is based on the \emph{truncations} and \emph{projections} of the boundary map: these are, for  $x\in \dg$, the  curves in the full flag manifold $\calF(X)$ of $X:=\quotient{E}{x^{d-k}}$ and $\calF(x^k)$ defined by
\begin{align*}
\pi_{x}: &\dg\backslash\{x\}\to \calF(x^k)\\
\pi_{x}(y):&= (y^{d-k+1}\cap x^k,\ldots,y^{d-1}\cap x^k).
\end{align*}
and
\begin{align*}
\pi_{[x]}:&\dg\backslash\{x\}\to \calF(X)\\
\pi_{[x]}(y):&= ([y^1+x^{d-k}]_X,\ldots,[y^{k-1}+x^{d-k}]_X),
\end{align*}
here $[\cdot]_X$ is the natural projection to $X$.

\begin{defn}
A representation $\r:\G\to \PGL(E)$ is  \emph{$k$--positive} if it is $\{1,\ldots, k\}$--Anosov and for any $x\in \dg$ the projection $\pi_{[x]}$ and the truncation $\pi_x$
are positive.
\end{defn}

In this definition we consider  positivity with respect to the positive structure on full flag manifolds alluded to before  (see Section \ref{s.5.1} and \cite{FG}). 
A representation $\rho$ is $1$--positive if and only if it is 1--Anosov, and it is 2--positive if and only if $\r$ and $\r^*$ are $(1,1,2)$--hyperconvex as in \cite{PSW1}. Since $k$--positivity is an open property in $\Hom(\G,\PGL(E))$ (Corollary  \ref{cor.k-pos open}), we get many examples as small deformations of Fuchsian loci:
%\begin{example}\label{ex.intro}\mB{Do we want this as an example, or rather in text?}
Let $\un d=(d_1,\ldots,d_j)$ be an ordered list of positive  integers, % $d_1\geq d_2 \geq \ldots\geq d_j$,  
  $\tau_{d_i}:\SL(\R^2)\to \SL(\R^{d_i})$  the $d_i-$dimensional irreducible representation of $\SL(\R^2)$,   $\tau_{\un d}=\tau_{d_1}\oplus \ldots\oplus \tau_{d_j}$, and  $\r_{hyp}:\G\to \SL(\R^2)$  a discrete and faithful representation. Then the composition $\tau_{\un d}\circ \r_{hyp}$ is $k$-positive for any $k$ with $2(k-1)<  (d_1-d_2)$ (cfr. Example \ref{e.Fuchpos}).

We prove that many features of Hitchin representations hold for $k$-positive representations:

\begin{itemize}
\item The boundary map $\xi^l$ has $C^1$ image  for  $l<k$ (Proposition \ref{prop.k-pos and Hk+Ck} and \cite[Proposition 8.11]{PSW1}).
\item The  $l$--collar lemma  holds for $l<k-1$ (Definition \ref{d.collar}, Corollary \ref{c.kpospos} and \cite[Theorem 1.3]{BP})
\item The representation $\r$ is $l$--positively ratioed for $l<k$ (Definition \ref{d.posrat}, Corollary \ref{c.kpospos}). Thus the $l$-th Finsler length function can be computed as intersection with a geodesic current  \cite{MZ}.
\item {{The boundary maps are hyperconvex:}} The sum
$x_1^{n_1}+\ldots+ x_l^{n_l}$
is direct for  $n_1+\ldots+ n_l=d$, $n_l\geq d-k$ and pairwise distinct $x_i\in\dg$  (Proposition \ref{prop.hyperconvex properties})

\end{itemize}
The only connected component of the $\PGL(E)$ character variety  consisting just of injective representations with discrete image is the Hitchin component, and all representations therein are irreducible. As a result it is possible to deform representations as in Example \ref{e.Fuchpos} to representations for which all Anosovness is lost. The first main theorem of the paper is that the Anosov and the $k$--positivity properties are lost in a very controlled way:

\begin{thmA}\label{thm.intro-degenerations}
Let $\{\r_n:\G\to \PGL(E)\}_{n\in \N}$ be $k$--positive and converge to $\r_0$ in $\Hom(\G, \PGL(E))$. Then the following hold:
\begin{enumerate}
\item $\r_0$ is $(k-3)$--positive.
\item If $\r_0$ is irreducible or $(k-1)$--Anosov, then $\r_0$ is $(k-1)$--positive.
\item If $\r_0$ is $k$--Anosov, then $\r_0$ is $k$--positive.
\end{enumerate}
\end{thmA}

In particular $k$--positive representations form connected components of the set of $k$--Anosov representations, and their irreducible limits can at worst lose Anosovness for the $k$-th root. This is not unexpected: while the first $k-1$ boundary maps of a $k$-positive representation have $C^1$ image, with derivative encoded by the two adjacent boundary maps, we have much less control on the last boundary map; as a result the $k$-th limit set might have fractal behaviour and much higher Hausdorff dimension. 

There are very few examples  of classes of Anosov representations whose limits are at least partially understood: notably this is the case for quasi-Fuchsian space, where active research in the past decades uncovered striking geometric phenomena, including doubly degenerate manifolds. In general, however, the mechanisms behind the loss of Anosovness is not known. Theorem \ref{thm.intro-degenerations} guarantees that degenerations of $k$--positive representations are relatively tame: if $k$ is bigger than 3  such degenerations are necessarily quasi-isometric embeddings, and we can also guarantee that they are Anosov for the first $k-3$ roots. Anna Wienhard suggested that, for this class, the Anosov properties that we can guarantee for the degenerations could be useful to better understand their geometry. For example, one could try to understand how the  $k$--Anosov property is lost by studying convergence of the associated domains of discontinuity \cite{Guichard-Wienhard-IM, KLP2}.

\medskip
As  we already remarked, $k$-positive representations are $l$-positvely ratioed for $l\leq k-1$. In the first step of the proof of Theorem \ref{thm.intro-degenerations}, we study limits of positively ratioed representations, and prove the following result, of independent interest:
\begin{thmA}\label{thmINTRO.limits of pos ratio general+proximal}
Let $\{\r_n:\G\to \PGL(E)\}_{n\in \N}$ be a sequence of $k$--positively ratioed representations converging to a $k$--proximal reductive representation $\r_0:\G\to \PGL(E)$. Then $\r_0$ admits a pair of equivariant transverse continuous $k$--positively ratioed boundary maps 
$$\xi^{k}_{\r_0}:\dg\to \Gr_k(E),\quad \xi^{d-k}_{\r_0}:\dg\to \Gr_{d-k}(E),$$ 
that are dynamics preserving at fixed points of elements with proximal image.

\end{thmA}
In particular a Zariski dense limit of positively ratioed representations is positively ratioed (Corollary \ref{cor:Zd}), and thus the set of positively ratioed representations is closed among Zariski dense representations. We will use Theorem \ref{thmINTRO.limits of pos ratio general+proximal} in the companion paper \cite{BP3} to show that $\Theta$-positive representations form  higher rank Teichm\"uller spaces, thus proving \cite[Conjecture 5.4]{GWpositivity}.
%%%%%%%%%%%%
\subsection*{Ideas in the proofs}
We conclude the introduction mentioning some of the crucial ideas to prove the main results. The real input needed for Theorem \ref{thmINTRO.limits of pos ratio general+proximal} is a technical transversality assumption, called \emph{condition} (Tr), that is implied by the proximality assumption on the limit, but is more general. {{Condition (Tr)  guarantees the existence of continuous equivariant transverse boundary maps (Theorem \ref{thm.limits of 1-pos ratio}).}} While we can not provide an explicit example of a sequence that does not satisfy condition (Tr), we do think it is likely that those  exist; see Section \ref{s.7.1} for an example in this direction. The proof of Theorem {{\ref{thm.limits of 1-pos ratio}}} %\ref{thmINTRO.limits of pos ratio general+proximal}
 is then  split in three major steps. First we combine the seed of transversality given by (Tr) and the positivity of the cross ratio to guarantee transversality of full orbits (Lemma \ref{lem opposition for limits of pos ratioed} and Proposition \ref{prop.limits of spr: transverse orbits}). Second, using again the positivity of the cross ratio, we show that such transverse orbits admit transverse left- and right-continuous extensions. Last, we use the conicality of the action of $\G$ on $\dg$ to show that the left- resp. right-continuous boundary maps need to be continuous.

We turn to the proof of Theorem \ref{thm.intro-degenerations}. The first step is to show that $k$--positive representations degenerate if and only if they lose some of the Anosovness (Corollary \ref{cor.par-pos closed among Anosov}). A crucial tool for this is a  powerful result by Labourie:  the boundary maps of strongly irreducible representations cannot collapse open sets (Proposition \ref{prop.labourie-irreducible}). As the example of the Fuchsian loci shows,  $k$--positive representations are not always irreducible. However, we prove that they are \emph{coherent}: a reductive $k$--positive representations $\r$ splits as $\r=\eta_1\oplus \eta_2\to \PGL(E_1\oplus E_2)$, where $\eta_1$ is strongly irreducible, and $\xi^j(\dg)\subset \Gr_j(E_1)$ for all $j=1,\ldots,k$ (Corollary \ref{cor.coherence}). Thus their boundary maps are also boundary maps of a strongly irreducible $k$--positive representation.

In the case of irreducible limits (Theorem \ref{thm.intro-degenerations} (2)),  we need to derive that the limiting representation $\r_0$ is Anosov up to $(k-1)$. We do this by induction showing that condition (Tr) is satisfied for $j<k$ if $\r_0$ is $(j-1)$--Anosov (Proposition \ref{prop.condition Tr strongly irred}), and thus obtain continuous equivariant transverse $j$--boundary maps (Theorem \ref{thm.limits of 1-pos ratio}). We then derive $j$--Anosovness thanks to the following proposition of independent interest, generalizing \cite[Proposition 4.10]{Guichard-Wienhard-IM}:
\begin{prop}[{Proposition \ref{lem.dynamics preserving for strongly irred}}]\label{prop.intro k-Ano}
	Let $\G$ be a word hyperbolic group and $\r:\G\to \PGL(E)$ be $(j-1)$--Anosov, strongly irreducible and admit a pair of equivariant, transverse continuous boundary maps 
	$$\xi^j:\di \G\to \Gr_j(E),\quad \xi^{d-j}:\di \G\to \Gr_{d-j}(E).$$
Then $\r$ is $j$--Anosov and $\xi^j,\xi^{d-j}$ are the Anosov boundary maps.
\end{prop}

For the study of general limits of $k$-positive representations, Theorem \ref{thm.intro-degenerations} (1), we use collar lemmas to guarantee that every element $\r_0(\g)$ for $\g\in \G$ is $l$--proximal for $l\leq k-2$. As all $\r_n$ are $l$--positively ratioed, we can apply Theorem \ref{thmINTRO.limits of pos ratio general+proximal} to derive that $\r_0$ admits continuous, dynamics preserving, transverse $l$--boundary maps. It remains to show that the eigenvalue gaps grow along unbounded sequences  (Definition \ref{defn.Anosov} (3)). For this we prove an additional  hyperconvexity property of $k$--positive representations (Corollary \ref{cor.strict convex for mixed quotient}) which leads to coherence of the boundary maps of the limiting representation $\r_0$. We believe that the idea behind the proof of Corollary \ref{cor.strict convex for mixed quotient} could be helpful in other concrete problems as well. Coherence of the boundary map allows to use Proposition \ref{prop.intro k-Ano} to derive Anosovness up to $k-3$.
Although we expect that this is true in general, we are not able to show that limits of $k$-positive representations are necessarily $(k-2)$--positive, because we  cannot guarantee any transversality in the limits of $\xi^{k-1}_{\r_n}$ and $\xi^{d-k+1}_{\r_n}$ in general, and thus we don't know how to ensure that the $(k-2)$-th eigenvalue gaps grow.

In Theorem  \ref{thm.intro-degenerations}(3) we assume that $\rho_0$ is $k$-Anosov. With this we can show that $\xi^{k-2}_{\r_n}$ converge pointwise to a continuous transverse boundary map. As $\xi^{k-1}_{\r_n}$ has $C^1$ image and the tangents are encoded by $\xi^{k-2}_{\r_n}$ and $\xi^{k}_{\r_n}$, we can show that the derivatives converge to a continuous map. We can thus apply a lemma form analysis that  guarantees that $\xi^{k-2}_{\r_0}$ and $\xi^{k}_{\r_0}$ are coherent and this allows to derive that  condition (Tr) is satisfied for $k-1$.

\subsection*{Structure of the paper} In Section \ref{s.boundary map} we introduce the technical transversality assumptions \emph{condition} (Tr) for limits of positively ratioed representations and show that if a converging sequence satisfies condition (Tr), then the limit representation admits a continuous equivariant transverse boundary map (Theorem  \ref{thm.limits of 1-pos ratio}). In Section \ref{s.pos-ratio} we conditions guaranteeing that condition (Tr) is satisfied and show that, in various instances, the existence of such a boundary map can be promoted to the limit being (almost) Anosov (cfr. Definition \ref{defn.Anosov}) and we prove Theorem \ref{thmINTRO.limits of pos ratio general+proximal}. In Section \ref{s.gaps} we discuss limits of representations satisfying property $H_k$ whose boundary maps $\xi^{k-1},\xi^{k+1}$ converge. This will be a crucial step in Theorem  \ref{thm.intro-degenerations}(3). In Section \ref{s.k-pos} we introduce $k$--positive representations and study their properties. In Section \ref{s.6} we  turn to the study of degenerations and prove Theorem  \ref{thm.intro-degenerations}. In the last section, Section \ref{s.ex}, we discuss open problems (e.g. to what extend Theorem  \ref{thm.intro-degenerations} is optimal) and examples illustrating those problems.

\subsection*{Acknowledgement} We thank Anna Wienhard for insightful conversations.

 \section{Preliminaries}
 \subsection{Notation}\label{s.notation}
In this paper we will consider projective linear representations of fundamental groups $\G$ of surfaces.  We list here some standard notation and conventions that we will keep throughout the paper.
 
\noindent{\it The group $\G$ and its boundary}
 \begin{itemize}
 \item  $\G$ will always denote a surface group, i.e. the fundamental group of a closed surface of genus at least 2. Its Gromov boundary $\dg$ is homeomorphic to the circle $\mathbb{S}^1$.
 
\item $\dg^{(j)}$ is the set of $j$--tuples of $\dg$ consisting of pairwise distinct points.
 
\item $(x,y)_z\subset \dg$ denotes the interval of $\dg\backslash\{x,y\}$ that does \emph{not} contain $z$, for $(x,y,z)\in \dg^{(3)}$ \cite{MZ}.
\end{itemize}

\noindent{\it The vector space $E$ and its subspace}

\begin{itemize}
\item  $E$ will always be a real vector space of dimension $d$.
 
\item If $E_1,E_2<E$ are subspaces, we denote by $\oplus$ the \emph{interior direct sum}, i.e. $E_1\oplus E_2$ if and only if $E_1+E_2=E$ and $E_1\cap E_2=\{0\}$.
\item  $\Gr_k(E)$ denotes the Grassmannians of $k$-planes in $E$. Given $V\in \Gr_k(E)$ and $W\in \Gr_{d-k}(E)$ we say that $V$ and $W$ are \emph{transverse} if $V\oplus W$, we then write $V\tv W$. If two subspaces are not transverse we write $V\ntv W$.

\end{itemize}

\noindent{\it Eigenspaces and eigenvectors}
\begin{itemize}

\item{{ $\lambda_1(g),\ldots,\lambda_d(g)\in \C$ denote the generalized eigenvalues of an element $g\in\GL(E)$ counted with multiplicity and ordered so that their absolute values are  non-increasing, i.e. $|\lambda_i(g)|\geq |\lambda_{i+1}(g)|$. 

\item If $g\in\PGL(E)$
 $$\frac{\lambda_i}{\lambda_{i+1}}(g):=\frac{\lambda_i(\tilde{g})}{\lambda_{i+1}(\tilde{g})}$$
for a non-trivial lift $\tilde{g}\in\GL(E)$. It
does not depend on the choice of lift.}}

\item If $g\in\PGL(E)$,  $E_{g}^{\lambda_k}$ is the generalized eigenspace for $\lambda_k(\tilde g)$ of any lift $\tilde g$ of $g$. Similarly $E_{g}^{\geq\lambda_k}$ denotes the sum of the generalized eigenspaces relative to eigenvalues with absolute value bigger or equal to $|\lambda_k|$. The spaces $E_{g}^{>\lambda_k}$, $E_{g}^{<\lambda_k}$, $E_{g}^{\leq\lambda_k}$ are defined analogously.

\item An element $g\in\PGL(E)$ is \emph{$k$-proximal} if it has a unique attracting point in $\Gr_k(E)$, equivalently if $|\frac{\lambda_k}{\lambda_{k+1}}(g)|>1$.

\end{itemize}

\noindent{\it Representations}
\begin{itemize}
\item A representation $\rho:\G\to\PGL(E)$  is \emph{$k$-proximal} if its image contains one $k$-proximal element.

\item Given a splitting $E_1\oplus E_2$ and representations $\eta_i:\G\to \PGL(E_i)$, we write $\r=\eta_1\oplus \eta_2\to\PGL(E_1\oplus E_2)$ if there are lifts such that $\tilde{\r}(\g) e_i =\tilde{\eta_i}(\g) e_i$ for all $e_i\in E_i$ and $\g\in \G$. %{\color{blue}In this case ratios of eigenvalues of $\eta_1$ and $\eta_2$ are well defined}\bea{What does this mean?}}}

\item We denote by $\rho^\dual:\G\to\PGL(E^*)$ the \emph{dual} or \emph{contragradient} representation of the representation $\rho:\G\to E$, which is defined by the equality.
$$(\rho^\dual(g)\cdot\omega)(v)=\omega(\rho(g)^{-1}\cdot v),\quad v\in E, \omega\in E^*.$$

 \end{itemize}

\subsection{Anosov representations}\label{s.Anosov}
Anosov representations were introduced by Labourie in \cite{Labourie-IM} for fundamental groups of negatively curved manifolds and generalized by Guichard-Wienhard to hyperbolic groups \cite{Guichard-Wienhard-IM}. There are by now many different characterizations (e.g. \cite{KLP,Bochi-Potrie-Sambarino}). For our purposes it will be convenient to choose the characterization \cite[Theorem 1.7]{GGKW} as a definition. For this we fix a word metric $|\cdot|_{\G}$ on the Cayley graph of $\G$ for a fixed finite generating set of $\G$. The \emph{stable length} of $\g\in\G$ with respect to this metric is $|\g|_{\infty}:=\lim_{n\to\infty} |\g^n|_{\G}\slash n$.

The definition of Anosov representations is based on properties of $\rho$-equivariant boundary maps 
$\xi^k:\dg\to \Gr_k(E)$. We say that such a boundary map is  \emph{dynamics preserving} at the attracting fixed point $\eta^+\in\dg$ of an infinite order element $\eta$ if $\rho(\eta)$ is $k$-proximal and $\xi^k(\g_{+})$ is the attracting fixed points for the action of $\r(\g)$ on $\Gr_k(E)$.

\begin{defn}\label{defn.Anosov}
A homomorphism $\r:\G\to \PGL(E)$ is called \emph{$k$-Anosov}  for $k=1,\ldots,d-1$, if there exist $\r$-equivariant continuous boundary maps $\xi^k:\dg\to \Gr_k(E)$, $\xi^{d-k}:\dg\to \Gr_{d-k}(E)$ so that
\begin{enumerate}
\item $\xi^k(x)$ and $\xi^{d-k}(y)$ are transverse for all $x\neq y\in \dg$
\item $\xi^k$ and $\xi^{d-k}$ are dynamics preserving at every attracting fixed point. 
\item $\left|\frac{\lambda_k}{\lambda_{k+1}}(\r(\g_i))\right|\to \infty$ if $|\g_i|_{\infty}\to\infty$.
\end{enumerate}
We say that $\rho$ is \emph{$k$-almost Anosov} if conditions $(1)$ and $(2)$  hold.
\end{defn}

\begin{remark}\label{rem ansov rep properties}For the following facts see e.g. \cite{Guichard-Wienhard-IM}.
	
\begin{enumerate}
\item An Anosov representation is discrete and faithful.
\item A $k$--Anosov representation is also $(d-k)-$Anosov.
\item The set of $k$--Anosov representations is open in $\Hom(\G,\PGL(E))$.
\end{enumerate}
\end{remark}
In the case of irreducible representations, the existence of boundary maps is enough to guarantee $1$-Anosovness:
\begin{prop}[{\cite[Proposition 4.10]{Guichard-Wienhard-IM}}]\label{prop.GW-irreducibility}
If $\r:\G\to\PGL(E)$ is irreducible and admits a pair of equivariant transverse continuous boundary maps $\xi^1:\dg\to\P(E),\;  \xi^{d-1}:\dg\to\Gr_{d-1}(E)$, then $\r$ is 1--Anosov with boundary maps  $\xi^1$,  $\xi^{d-1}$.
\end{prop}

\begin{notation}
We will often use the shortcuts $x_{\r}^k$ and  $x^k$  for $\xi^k(x)$ (here $\r$ is a $k$--Anosov representation and $\xi^k$ is the associated boundary map).
Similarly we sometimes write $\g_{\r}$ instead of $\rho(\g)$ for  $\g\in \G$.
\end{notation}

A standard method to obtain Anosov representations is to deform representations in the so-called \emph{Fuchsian loci}, i.e. representations obtained by composing holonomies of hyperbolizations with $\SL(\R^2)$-representations. We fix here the notation that will be used throughout the paper:

\begin{example}[Fuchsian Loci]\label{ex.Fuchsian I}
Let $\tau_{d_i}:\SL(\R^2)\to \SL(\R^{d_i})$ be the $d_i-$dimensional irreducible representation\footnote{this is uniquely defined up to conjugation} of $\SL(\R^2)$ and set $\tau_{(d_1,\ldots,d_j)}:=\tau_{d_1}\oplus \ldots\oplus \tau_{d_j}$ for positive  integers $d_1\geq d_2 \geq \ldots\geq d_j$.
A \emph{Fuchsian locus}, or the \emph{$\un d$-Fuchsian locus} for the specified multiindex $\un d=(d_1,\ldots,d_j)$, is the set of representations obtained as {{composition (and projectivization) $\tau_{\un d}\circ \r_{hyp}:\G\to\PGL(\R^{d_1}\oplus \ldots\oplus\R^{d_j})$}} as $\r_{hyp}:\G\to \SL(\R^2)$ varies in the  set of  discrete and faithful representations. Such representations are  {{$\{1,\ldots, k\}$-Anosov for $k<\frac{1}{2}(d_1-d_2)+1$}}, hence the same holds for all small deformations in $\Hom(\G,\PGL(E))$. Representations in  Fuchsian loci are self-dual, namely they are conjugated to their dual representation.
\end{example}

\subsection{Strong irreducibility}
A linear representation $\rho:\G\to\PGL(E)$ is \emph{strongly irreducible} if the restriction of $\rho$ to any finite index subgroup $\G'<\G$ is irreducible.  
The next important property of strongly irreducible representations is due to Labourie. We give here a slightly modified argument that doesn't require continuity, which will be useful for our purposes.
\begin{prop}[{\cite[Lemma 10.2]{Labourie-IM}}]\label{prop.labourie-irreducible}
Let $\r:\G\to\PGL(E)$  be a representation. If there exists $x\in\dg$, an equivariant  boundary map $\xi^k:\G x\to\Gr_k(E)$, an open non-empty set $U\subset \dg$ and a subspace $V\in\Gr_{d-k}(E)$ such that 
$$\forall y\in\G x\cap U\quad \dim(\xi^k(y)\cap V)\geq 1,$$
then $\r$ is not strongly irreducible.
\end{prop}
\begin{proof}
Thanks to \cite[Proposition 10.3]{Labourie-IM}, if there exists a subspace $W\in\Gr_{d-k}(E)$ such that 
\begin{equation}\label{e.Lab}
	\forall y\in\G x\quad \dim(\xi^k(y)\cap W)\geq 1,
\end{equation}
then the connected component  of the identity $G$ in the Zariski closure of $\G$ doesn't act irreducibly, and thus the restriction of $\rho$ to the finite index subgroup $\G'=\G\cap G<\G$ is not irreducible. We will show the existence of such subspace under the assumptions of the proposition.

Consider an hyperbolic element $\g$ in $\G$ such that $\g^+\notin \G x$ and $\g^-\in U$. Up to possibly shrinking $U$ we can assume that 
$$\g^nU\subset \g^{n+1}U\quad\text{ and }\quad\bigcup_{n\in\N} \g^nU=\dg\setminus\{\g^+\}.$$
Set $V_n:=\rho(\g^{n})V$. By construction, for every $y\in \g^nU\cap \G x$,  $\dim(\xi^k(y)\cap V_n)\geq 1$. Since the subset of $\Gr_{d-k}(E)$ consisting of subspaces $B$ such that $\dim(\xi^k(y)\cap B)\geq 1$ is closed, Equation \eqref{e.Lab} holds for any accumulation point $W$ of the sequence $V_n$.
\end{proof}
For 1--proximal representations  irreducibility already implies strong irreducibility:
\begin{lem}[{\cite[Lemma 5.12]{Guichard-Wienhard-IM}}]\label{lem.GW strong irreducibility}
Let $\r:\G\to\PGL(E)$ be irreducible and $1$--proximal. Then $\r$ is strongly irreducible.
\end{lem}

%%%%%%%%
\subsection{Cross ratios}\label{s.cr}
In the paper we will use the classical  cross ratios defined on the following set of quadruples of points in Grassmannians:
\begin{align*}
\calA_k\: :=\{ (V_1,W_2,W_3,V_4) | &V_i\in \Gr_k(E), W_j\in \Gr_{d-k}(E) \text{ and } V_j\tv W_i\\
 &\text{ for } (j,i)=(1,2),(4,3) \text{ or } (j,i)=(1,3), (4,2)\}.
\end{align*}

\begin{defn}
Let $(V_1,W_2,W_3,V_4)\in\calA_k$. The (generalized) cross ratio $\cro_k:\calA_k\to \R\cup\{\infty\}$ is defined by
\begin{align*}
\cro_k (V_1,W_2,W_3,V_4):=\frac{V_1\wedge W_3}{V_1\wedge W_2} \frac{V_4\wedge W_2}{V_4\wedge W_3},
\end{align*}
where $V_i\wedge W_j$ denotes the element $v_1\wedge \ldots\wedge v_k \wedge w_1\wedge \ldots \wedge w_{d-k}\in \wedge^d E \simeq \R$ for bases $(v_1,\ldots, v_k),(w_1,\ldots, w_{d-k})$  of $V_i$ and $W_j$, respectively, and a fixed identification $\wedge^d E \simeq \R$. We use the convention $\frac{a}{0} :=\infty$ for any non-zero $a\in\R$. The value of $\cro_k$ is independent of all choices made.
\end{defn}

\begin{prop}\label{prop.property of grassmannian cro}
Let $V_1,V_4,V_5\in\Gr_k(E)$ and $W_2,W_3,W_5\in\Gr_{d-k}(E)$. Whenever all quantities are defined we have
\begin{enumerate}
\item $\cro_k (V_1,W_2,W_3,V_4)^{-1}=\cro_k (V_4,W_2,W_3,V_1)=\cro_k (V_1,W_3,W_2,V_4)$
\item $\cro_k (V_1,W_2,W_3,V_4)\cdot\cro_k (V_4,W_2,W_3,V_5)=\cro_k (V_1,W_2,W_3,V_5)$
\item $\cro_k (V_1,W_2,W_3,V_4)\cdot\cro_k (V_1,W_3,W_5,V_4)=\cro_k (V_1,W_2,W_5,V_4)$
\item $\cro_k (V_1,W_2,W_3,V_4)=0\Longleftrightarrow V_1 \ntv W_3$ or $V_4\ntv W_2$
\item $\cro_k (V_1,W_2,W_3,V_4)=\infty\Longleftrightarrow V_1 \ntv W_2$ or $V_4\ntv W_3$
\item $\cro_k (V_1,W_2,W_3,V_4)=\cro_k (g V_1,g W_2,g W_3,g V_4)\quad  \forall g\in\PGL(E)$.
\end{enumerate}
The identities $(2)$ and $(3)$ will be called \emph{cocycle identities}.
\end{prop}

{{If $g\in \PGL(E)$ is $\{k,(d-k)\}$--proximal, we denote by $g_+^k\in\Gr_k(E)$ the span of its first $k$ generalized eigenspaces and by $g_-^k:=(g^{-1})_+^k\in\Gr_k(E)$.
}}
The following is straight forward to check:
\begin{lem}\label{lem.cro-period}
If $g\in \PGL(E)$ is $\{k,(d-k)\}$--proximal, then for every $V\in \Gr_{k}(E)$ transverse to $g^{d-k}_\pm$, and $W\in \Gr_{d-k}(E)$ transverse to $g^{k}_\pm$ it holds
$$\cro_k (g_-^k,W,gW,g_+^{k})=\cro_k (V,g_-^{d-k},g_+^{d-k},gV)=\frac{\lambda_1(g)\cdots\lambda_k(g)}{\lambda_d(g)\cdots\lambda_{d-k+1}(g)}.$$ 
\end{lem}

\begin{lem}\label{lem.cross ratio is algebraic}
The cross ratio is algebraic on pairwise transverse 4-tuples and extends continuously to its maximal domain of definition.
\end{lem}

\begin{proof}
The continuity at quadruples with image equal to infinity follows e.g. from \cite[Lem 3.7, Ex 3.11]{Beyrer}.
\end{proof}

We will denote in the following by $\pcr$ the standard projective cross-ratio on $\R\P^1$. While this agrees with $\gcr_1$ in case $\dim E=2$, we find it useful to use a different notation since in our paper the vector space $E$ has dimension $d$, and since the projective cross-ratio has additional symmetries than $\gcr$. We nevertheless have the following well known identities.

{{
		\begin{remark}\label{rem.cross ratio as projective} 
			Let $Q_1\neq Q_2\in \Gr_{d-1}(E)$ and $P_1,P_2\in \P(E)\setminus \P(Q_1\cap Q_2)$. For the natural projections $[P_1],[Q_1],[Q_2],[P_2]\in\P(\quotient{E}{Q_1\cap Q_2})\simeq \R\P^1$ we have
			$$\cro_1(P_1,Q_1,Q_2,P_2)=\pcr([P_1],[Q_1],[Q_2],[P_1])$$
			provided one of the sides is defined. 
			Similarly, if $P_1\neq P_2\in \P(E)$ and the cross ratio is defined, then it can be written as
			$$ \cro_1(P_1,Q_1,Q_2,P_2)=\pcr(P_1,Q_1\cap P,Q_2\cap P,P_2)$$
			where $P=P_1\oplus P_2\simeq \R^2$.
\end{remark}}}

Additionally, for some degenerate configurations, we are able to compute a cross ratio in terms of the standard projective cross-ratio in an associated projective line. This will be very useful in the paper.

\begin{prop}[{\cite[Corollary 3.12]{Beyrer-Pozzetti}}]\label{prop.projection of cross ratio}
Let $P^k,Q^k\in \Gr_k(E)$ be such that $\dim P^k\cap Q^k=k-1$. Denoting by $X^{k+1}=\langle P^k,Q^k\rangle$, $X:=\quotient{X^{k+1}}{P^k\cap Q^k}\simeq \R^2,$ it holds 
$$\gcr_k\left(P^k,\,S^{d-k},T^{d-k},Q^k\right)=\pcr([P^k]_X, [S^{d-k}\cap X^{k+1}]_X,[T^{d-k}\cap X^{k+1}]_X,[Q^k]_X),$$
whenever one of the sides is defined. 
\end{prop}

Martone--Zhang used  cross ratios to introduce the class of positively ratioed representations. We use here a slightly stronger asymmetric definition (see \cite[Remark 5.2]{Beyrer-Pozzetti} for a comparison)
\begin{defn}\label{d.posrat} (cfr. \cite[Def. 2.25]{MZ}, \cite[Def. 5.1]{Beyrer-Pozzetti})
A representation $\r:\G\to \PGL(E)$ is  \emph{$k$--positively ratioed} if it is $k$--Anosov and for every cyclically ordered\footnote{We say that a $n$-tuple in $\dg$ is \emph{cyclically ordered} if it is in clockwise or counter clockwise order} quadruple $(x,y,z,w)$ of points in $\dg$ 
\begin{align}\label{eq.def pos ratio}
\cro_k(x^k,y^{d-k},z^{d-k},w^k)\geq 1.
\end{align}
If $\r$ is only $k$--almost Anosov, we call $\r$ $k$--almost positively ratioed.
\end{defn}

It is shown in \cite{MZ} that the inequalities in \eqref{eq.def pos ratio} are necessarily  strict. The proof applies verbatim to almost positively ratioed representations.

%%%%%%%%%%
\subsection{Property $H_k$}
%%%%%%%%%5
The following property was introduced by Labourie in his work on Hitchin representations, and used in \cite{PSW1, ZZ} to study differentiability properties of limit curves. The following definition works also in the case $k=1,d-1$: every representation is considered being $0$-Anosov and $d$-Anosov, we further set, for $w\in\dg$, $w^0=\{0\}$ and $w^d=E$.
\begin{defn}[Compare {\cite[Section 7.1.4]{Labourie-IM}}]\label{d.Hk}
A representation $\rho:\Gamma\to\PGL(E)$ satisfies \emph{property $H_k$} for $1\leq k \leq d-1$ if  it is $\{k-1,k,k+1\}$--Anosov and for all $(x,y,z)\in\dg^{(3)}$ the following sum is direct
\begin{align*}%\label{eq property H}
x^k+ \left(y^k\cap z^{d-k+1}\right)+ z^{d-k-1}.
\end{align*} 
we say that $\r$ satisfies \emph{property $H^*_k$} if $\r$ satisfies properties $H_k$ and $H_{d-k}$. 

If $\r$ is $\{k-1,k,k+1\}$--almost Anosov and the above transversality holds, we say that \emph{its boundary maps satisfy property $H_k$ (resp. $H_k^*$)}.
\end{defn}

In \cite[Proposition 4.4]{Beyrer-Pozzetti} it was observed that $\r$ satisfies property $H_k$ if and only if $\r^{\dual}$ satisfies $H_{d-k}$. 

Representations satisfying property $H_k$ have the following two properties that will be crucial for us:

\begin{prop}[{\cite[Proposition 8.11]{PSW1}}]\label{prop.1-hyperconvex for property Hk} If $\rho:\Gamma\to\PGL(E)$ satisfies property $H_k$, then the boundary curve $\xi^k$ has $C^1-$image and the tangent space is given by
\begin{align*}
T_{x^k}\xi^k(\dg)=\{ \phi\in \Hom(x^k,y^{d-k})| x^{k-1}\subseteq \ker \phi, \;{\rm Im}\: \phi\subseteq  x^{k+1}\cap y^{d-k}  \}
\end{align*}
for any $y\neq x \in\dg$.
\end{prop}

In \cite{Beyrer-Pozzetti} we used such good control on the derivative to show that  representations satisfying property $H^*_k$ are positively ratioed:
\begin{prop}[{\cite[Theorem 1.3]{Beyrer-Pozzetti}}] \label{prop.hyperconvex implies positively ratioed}
If $\r:\G\to\PGL(E)$ satisfies property $H^*_k$, then $\r$ is $k$--positively ratioed.
\end{prop}

%%%%%%
\subsection{Property $C_k$}
%%%%%%%
We will also need a second transversality condition on boundary maps. Also in this case the definition makes sense also in the case $k=1, d-2$: indeed by convention every representation is 0-Anosov and $d$-Anosov, and we set $x^d=E$ and $x^0=\{0\}$. 
\begin{defn} [{\cite[Definition 1.6]{Beyrer-Pozzetti}}]
A representation $\r:\G\to \PGL(E)$ satisfies \emph{property $C_k$} for $1\leq k \leq d-2$ if it is $\{k-1,k,k+1,k+2\}-$Anosov and for $(x,y,z)\in \dg^{(3)}$ the sum 
$$x^{d-k-2} + (x^{d-k+1}\cap y^{k})+ z^{k+1}$$ 
is direct. We say that $\r$ satisfies property $C_k^*$ if it satisfies property $C_k$ and $C_{d-k-1}$.
\end{defn}

The representation $\r$ satisfies $C_k$ if and only if $\r^{\dual}$ satisfies $C_{d-k-1}$ \cite{Beyrer-Pozzetti}.
A first important consequence of property $C_k$ is the following

\begin{prop}[{\cite[Proposition 4.15]{Beyrer-Pozzetti}}]\label{prop.Ck yields convex curve}
If $\r:\G\to \PGL(E)$ satisfies property $H_k$ and $C_k$, then for any $x\in\dg$
$$\begin{array}{rl}
\sfp_x:&\dg\to \P(\quotient{x^{d-k+1}}{ x^{d-k-2}})\simeq \R\P^2\\
&\left\{\begin{array}{l}
\sfp_x(y):= [y^k\cap x^{d-k+1}] ,\quad y\neq x\\
\sfp_x(x):= [x^{d-k-1}]
\end{array}\right.
\end{array}$$
has  $C^1$--image with tangents % at $\sfp_x(y)$ given by 
$\sfp_x^{(2)}(y):=[y^{k+1}\cap x^{d-k+1}]$ and parametrizes the boundary of a strictly convex domain.
\end{prop}

The transversality properties $C_k$ and $H_k$ guarantee that a representation  satisfies a collar lemma for suitable length functions. Two elements $g,h\in\G=\pi_1(S)$ are \emph{linked}, if the corresponding closed geodesics with respect to some (and thus any) hyperbolic metric intersect in $S$. We are interested in collar lemmas comparing the root and the weight, as this is in many cases optimal (see \cite[Appendix]{Beyrer-Pozzetti}), but still ensures proximality in the limit (Proposition \ref{lem.almost Anosov limitcollar}). 
\begin{defn}\label{d.collar}
A representation $\rho:\G\to\PGL(E)$ satisfies the $k$-th \emph{root versus weight collar lemma}, or simply the $k$--collar lemma, if 
for any linked pair $g,h\in \G$ it holds
\begin{align*}
	\frac{\lambda_1\ldots\lambda_k}{\lambda_d\ldots\lambda_{d-k+1}}(\rho(g))>\left(1-\frac{\lambda_{k+1}}{\lambda_k}(\rho(h))\right)^{-1}.
\end{align*}
\end{defn}	
We have
\begin{thm}[{\cite[Theorem 1.3]{Beyrer-Pozzetti}}]\label{thm.collar lemma}
For $k$ such that $1\leq k \leq d-1$ let $\r:\G=\pi_1(S)\to \PGL(E)$ be an Anosov representation satisfying properties $H^*_{k-1},H^*_{k}, H^*_{k+1}$ and $C_{k-1}$, $C_k$. Then $\rho$ satisfies the $k$--collar lemma.\footnote{If $k=1$ then $H^*_{k-1}$ and $C_{k-1}$ are not defined (and not necessary) and if $k=d-1$ then $H^*_{k+1}$ and $C_{k}$ are not defined (and not necessary)}
\end{thm}

%%%%%%%%
\subsection{Semi-simplification}
%%%%%%%%%
A representation is \emph{reductive} if the Zariski closure of its image is a reductive subgroup of $\PGL(E)$.
It is well known that every reductive representation $\r:\G\to\PGL(E)$ splits as direct sum $\r=\eta_1 \oplus \ldots \oplus \eta_i:\G\to \PGL(E_1\oplus \ldots \oplus E_i)$ of irreducible representations. For those representations the following lemma is very useful. The proof applies also to almost Anosov representations.

\begin{lem}[{\cite[Lemma 4.17]{Beyrer-Pozzetti}}]\label{lem.splitting boundary map for reducible reps}
Let $\rho=\eta_1\oplus \eta_2:\G\to\PGL(E_1\oplus E_2)$. If $\rho$ is $k$--(almost) Anosov, then the dimensions $k_i$ of the intersection $x_i^{k_i}:=x^{k}\cap E_i$ are constant, $x^k=x_1^{k_1}\oplus x_2^{k_2}$, and $\eta_i$ is $k_i$--(almost) Anosov with boundary map $x\mapsto x_i^{k_i}$. 
\end{lem}

In the context of  representations satisfying property $H_k$ we recall

\begin{prop}[{\cite[Proposition 4.18]{Beyrer-Pozzetti}}]\label{prop.Hk and reducible}%
Let $\rho=\eta_1\oplus \eta_2:\G\to\PGL(E_1\oplus E_2)$ be $\{k-1,k,k+1\}$--(almost) Anosov. Assume, without loss of generality, that $(k-1)_1=k_1-1$. Then the boundary maps associated to $\rho$ have property $H_k$ if and only if $(k+1)_1=k_1+1$ and the boundary maps associated to $\eta_1$ have property $H_{k_1}$. 
\end{prop}
Given a representation $\r:\G\to \PGL(E)$ it will be often convenient to study a {semi-simplification} $\r^{ss}:\G\to\PGL(E)$. Let $H$ be the Zariski closure of $\r(\G)$ and $H=L\ltimes U$ a splitting as a semidirect product of a Levi factor $L$ and the unipotent radical $U$. A \emph{semi-simplification} $\r^{ss}$ is the projection of $\r$ to the Levi factor $L$. Observe that $\r^{ss}$ is well defined up to conjugation. 
Equivalently, a semi-simplification is a representation in the unique closed orbit inside $\ov{\PGL(E)\cdot \r} \subset\Hom(\G,\PGL(E))$, where the action of $\PGL(E)$ is by conjugation.

By definition any semi-simplification is reductive and one can reduce the study of properties of $\r$ to studying the same properties of $\r^{ss}$, and use that $\r^{ss}$ is completely reducible. This is also useful when considering the Anosov property:

\begin{prop}[{\cite[Proposition 2.39]{GGKW}}]\label{prop.semi-simplification}
A representation $\r:\G\to\PGL(E)$ is $k$--Anosov if and only if $\r^{ss}$ is.
\end{prop}

For studying limits we observe: if $\{\r_n\}_{n\in \N}$ converges to $\r_0$ in $\Hom(\G,\PGL(E))$, then we can find conjugates such that $\{g_n\r_n g_n^{-1}\}_{n\in \N}$ converges to $\r^{ss}_0$.
\subsection{Exterior power of representations}\label{s.ext}
We denote by {{$\wedge^k:\PGL(E)\to\PGL(\wedge^k E)$ the homomorphism induced by the diagonal action of $\GL(E)$ on $\wedge^k E$.}}
The \emph{$k$--th exterior power}  $\wedge^k\r:\G\to\PGL(\wedge^k E)$ of a representation $\r:\G\to \PGL(E)$ is the composition of $\r$ with the homomorphism $\wedge^k$. Observe that the exterior power commutes with the semi-simplification: $(\wedge^k\r)^{ss}=(\wedge^k\r^{ss})$.

The exterior power allows to reduce the study of $k$--Anosov representations to the study of $1$--Anosov representations:
\begin{prop}\label{prop.Anosov for exterior power}
A representation $\r:\G\to\PGL(E)$ is $k$--(almost) Anosov if and only if $\wedge^k\r$ is 1--(almost) Anosov. If $\r$ is $\{k-1,k+1\}$--(almost) Anosov, then $\wedge^k \r$ is $2$--(almost) Anosov. 
\end{prop}
\begin{proof}
Given $g\in\GL(E)$ with generalized eigenvalues $\lambda_1,\ldots,\lambda_d$, the eigenvalues of $\wedge^k g$ are given by $\{\lambda_{i_1} \ldots \lambda_{i_k} | 1\leq i_1< \ldots < i_k\leq k\}$. Thus we have
\begin{align*}
\left| \frac{\lambda_1}{\lambda_2}(\g_{\wedge^k\r})\right|=\left|\frac{\lambda_k}{\lambda_{k+1}}(\g_{\r})\right|, \quad \left|\frac{\lambda_2}{\lambda_3}(\g_{\wedge^k\r})\right|=\min \left\{\left|\frac{\lambda_{k-1}}{\lambda_k}(\g_{\r})\right|,\left|\frac{\lambda_{k+1}}{\lambda_{k+2}}(\g_{\r})\right|\right\}.
\end{align*}
In particular $g$ is $k$-proximal if and only if $\wedge^kg$ is one proximal.

Given $x^k\in\Gr_k(E)$ we set  $\wedge^k x^k\in \P(\wedge^k E)$ to be the projective class of $x_1\wedge\ldots\wedge x_k$ - this does not depend on the choice of the basis $x_1,\ldots,x_k$ of $x^k$. This yields a smooth embedding $\wedge^k:\Gr_k(E)\to \P(\wedge^k E)$. If $\xi^k:\dg\to\Gr_k(E)$ and $\xi^{d-k}:\dg\to\Gr_{d-k}(E)$ are equivariant, transverse, continuous, dynamics preserving, then also
$$\wedge^k\xi^k:\dg\to\P(\wedge^k E),\quad \wedge^{d-k}\xi^{d-k}:\dg\to\P(\wedge^{d-k} E)\simeq \P(\wedge^k E^*)$$
are  equivariant, transverse, continuous, dynamics preserving - here we (naturally) identify $\P(\wedge^k E^*)\simeq \Gr_{d_k-1}(\wedge^k E)$, where $d_k=\dim\wedge^k E$.

Vice versa, since $\wedge^k(\Gr_k(E))\subset\P(\wedge^k E)$ is a closed submanifold, and for any $k$-proximal $g$ the attracting fixed point of $\wedge^kg$ belongs to $\wedge^k(\Gr_k(E))$, if $\wedge^k\rho$ admits a continuous, equivariant, dynamics preserving map $\xi_\wedge$, then $\xi_\wedge$ has the form $\wedge^k\xi^k$ for some $\rho$-equivariant, transverse, continuous dynamic preserving map $\xi^k$.

Given a nested pair $y^{k-1}\subset y^{k+1}$, we choose a basis $y_1,\ldots, y_{k+1}$  of $y^{k+1}$ such that $y_1,\ldots, y_{k-1}$ is a basis of $y^{k-1}$. We define 
$$\wedge^{k-1} y^{k-1}\wedge y^{k+1}:=\spa \{y_1\wedge \ldots \wedge y_k,y_1\wedge  \ldots \wedge y_{k-1}\wedge y_{k+1}\}\in \Gr_2(\wedge^k E).$$ 
If $\xi^j:\dg\to\Gr_j(E)$ for $j=k-1,k+1,d-k-1,d-k+1$ are  equivariant, transverse, continuous and dynamics preserving boundary maps, then
\begin{align*}
	\wedge^{k-1} \xi^{k-1} \wedge \xi^{k+1}&:\dg \to \Gr_2(\wedge^k E),\\
	\wedge^{d-k-1} \xi^{d-k-1} \wedge \xi^{d-k+1}&:\dg \to \Gr_2(\wedge^{d-k} E)\simeq \Gr_{d-2}(\wedge^k E)
\end{align*}
are also equivariant, transverse, continuous and dynamics preserving maps.
\end{proof}

The following fact follows from the definition of the cross ratio:
\begin{prop}\label{prop.pos ratio exterior power}
The representation $\r:\G\to\PGL(E)$ is $k$--(almost) positively ratioed if and only if $\wedge^k \r$ is $1$--(almost) positively ratioed.
\end{prop}

In a similar way property $H_k$ and property $H_1$ are related.

\begin{prop}[{\cite[Proposition 8.11]{PSW1}}]\label{prop.Hk implies H1 for exterior power}
	The boundary maps of the $\{k-1,k,k+1\}$--almost Anosov representation $\r:\G\to\PGL(E)$ satisfy property $H_k$ if and only if the boundary maps of  $\wedge^k\r$ satisfy property $H_1$.
\end{prop}

\section{Boundary maps for limit representations}\label{s.boundary map}

In this section we study limits $\rho_0$ of sequences $(\rho_n)$ of $k$-positively ratioed representations. We show that, under a  technical mild transversality assumption, condition (Tr), such limits admit continuous, equivariant, transverse boundary maps (Theorem \ref{thm.limits of 1-pos ratio}).  We construct the boundary maps for $\rho_0$ by choosing two points $\bx,\by\in\dg$, and assuming, up to extracting a subsequence of $(\rho_n)$, that the sequences $\xi_{\rho_n}^k(\bx)$ and $\xi_{\rho_n}^{d-k}(\by)$ converge. Condition (Tr) encodes transversality and visibility properties of the orbits of these two points. We then show that, under condition (Tr), the boundary maps defined on the orbits of $\bx$ and $\by$ extend continuously to provide the desired boundary maps.

\begin{comment}
We prove Theorem \ref{thmINTRO.limits of pos ratio general+proximal} in 4 steps.
In Section \ref{s.trans} we show how  one can use the positivity of the cross ratio to guarantee full transversality along  orbits (Proposition \ref{prop.limits of spr: transverse orbits}). In Section \ref{s.L} we introduce a technical condition, condition (Tr), on pointwise limits: 
it encodes the technical requirements needed to guarantee existence of equivariant boundary maps in the limit. In Section \ref{s.bdry} we prove that  left and right continuous limits of pointwise limits give continuous positively ratioed maps (Theorem \ref{thm.limits of 1-pos ratio}). In Section \ref{s.conclusion} we conclude the proof by showing that such boundary maps are dynamics preserving at fixed points of proximal elements. We conclude the section giving additional conditions that allow to derive that the limit representation $\r_0$ in Theorem \ref{thmINTRO.limits of pos ratio general+proximal} is $k$--positively ratioed.
\end{comment}
\begin{notation}
In the whole section,{{ if not stated otherwise,}} $\{\r_n:\G\to \PGL(E)\}_{n\in \N}$  denotes  a sequence of $k$--positively ratioed representations converging to a representation $\r_0$ in $\Hom(\G,\PGL(E))$.
\end{notation}

\subsection{Condition (Tr) and transversality}\label{s.L}

We choose two points $\bx,\by\in\dg$.  Since  $\Gr_k(E)$ is compact,  up to extracting a subsequence, we can assume that the limits  $\xi^k_{\r_{n}}(\bx)$  and $\xi^{d-k}_{\r_{n}}(\by)$ exist. We will denote by 
$$\xi_{\bx}^k:\G \bx\to \Gr_k(E)\qquad \xi_{\by}^k:\G \by\to \Gr_k(E)$$ the induced $\rho_0$-equivariant boundary map; since the action of of $\PGL(E)$ on $\Gr_k(E)$ is continuous, it furthermore holds $\xi_{\bx}^k(\g\bx)=\lim\xi_{\rho_n}^k(\g\bx)$, and analogously for $\by$. 
The following technical conditions on the maps $\xi_{\bx}^k$ and $\xi_{\by}^{d-k}$ will guarantee that they can be extended to continuous, $\rho_0$-equivariant, transverse, boundary maps defined on the whole $\dg$. Condition (Tr-1) is a mild transversality condition, while Condition (Tr-2) are visibility conditions: cross-ratios can be used to separate points in the closure of the orbits. We will see in Section \ref{s.pos-ratio} that in many cases it is possible to verify that condition (Tr) holds for well chosen pairs of points $\bx,\by\in\dg$.

\begin{defn}
	The sequence $\{\r_n\}$  satisfies\footnote{When more boundary maps are involved, we say that condition (Tr) is \emph{satisfied for the $k$--boundary map}.} \emph{condition $\rm{(Tr)}$ with respect to $\bx,\by\in \dg$} if $\xi^k_{\r_n}(\bx)$ (resp. $\xi^{d-k}_{\r_n}(\by)$) converge and the following 
	properties hold:
	\begin{enumerate}
		\item[(Tr-1)] For every $l\in \overline{\xi^k_{\bx} (\G\bx)}$ and every $p\in \overline{\xi^{d-k}_{\by} (\G\by)}$ there exists $\g\in\G$ such that $l$ is transverse to $\r_0(\g) p$.
		\item[(Tr-2a)] If $\{x_i\},\{x'_i\}\subset \G \bx$ are sequences converging  to $x_0\in\dg$ and the limits $l:=\lim \xi^k_{\bx}(x_i)$, $l':=\lim \xi^k_{\bx}(x'_i)$  exist and are distinct, then for every open set $U$ in $\dg$ we find $z,z'\in U\cap \G \by$ with 
		$$\cro_k(l,\xi^{d-k}_{\by}(z),\xi^{d-k}_{\by}(z'),l')\neq 1.$$ 
		\item[(Tr-2b)]The same holds with roles of $\xi^k_{\bx}$ and $\xi^{d-k}_{\by}$ reversed.
	\end{enumerate}
\end{defn}

The first goal of the section is to show that the seed of transversality of (Tr-1) already guarantees that almost all pairs in $\overline{\xi^k_{\bx} (\G\bx)}\times \overline{\xi^{d-k}_{\by} (\G\by)}$ are transverse:
\begin{prop}\label{cor.transversality in the closure}
	If condition {\rm(Tr-1)} holds with respect to $\bx,\by\in \dg$, then $l\tv p$ for all $l=\lim \xi_{\bx}^k(x_i)$ and $p=\lim \xi_{\by}^{d-k}(y_i)$ with $\lim x_i\neq \lim y_i$.
	\end{prop}

As a preliminary lemma in the proof of Proposition \ref{cor.transversality in the closure} observe that a diagonal sequence argument let us  write the points $l,p$ as limit of values of the boundary maps for $\rho_n$:
\begin{lem}\label{lem.diagonal-sequence}
For every $p=\lim \xi_{\by}^{d-k}(y_i)\in \overline{\xi^{d-k}_{\by} (\G\by)}$, $l=\lim \xi_{\bx}^k(x_i)\in \overline{\xi^{k}_{\by} (\G\bx)}$, there exists a subsequence $(n_i)$ such that 
	$l=\lim_{i\to\infty}\xi^k_{\rho_{n_i}}(x_i)$ and $p=\lim_{i\to\infty}\xi^{d-k}_{\rho_{n_i}}(y_i)$.
\end{lem}

Recall from Section \ref{s.cr} that $\calA_{k}\subset \Gr_k(E)\times\Gr_{d-k}(E)\times\Gr_{d-k}(E)\times \Gr_k(E) $ is the open set of definition of the cross ratio $\cro_{k}$. The key property of the cross ratio that will let us prove Proposition \ref{cor.transversality in the closure} is:

\begin{lem}\label{lem opposition for limits of pos ratioed} 
If $(p_n,u_n,v_n,q_n)\in\calA_{k}$ satisfy $\cro_{k} (p_n,u_n,v_n,q_n)\geq 1$ and converge to  $(p_0,u_0,v_0,q_0)$ with $p_0\tv u_0$ and $q_0\tv v_0$, then
$p_0\tv v_0,\quad q_0\tv u_0$, and $\quad \cro_{k} (p_0,u_0,v_0,q_0)\geq 1.$
\end{lem}

\begin{proof}
The transversality $p_0\tv u_0$ and $q_0\tv v_0$ implies that $(p_0,u_0,v_0,q_0)\in \calA_{k}$. Since $\cro_{k}$ is continuous on its domain of definition, we deduce that $\cro_{k}(p_0,u_0,v_0,q_0)\geq 1$ or $\cro_{k}(p_0,u_0,v_0,q_0)=\infty$. If $p_0\ntv v_0$ or $q_0\ntv u_0$, then $\cro_{k}(p_0,u_0,v_0,q_0)=0$. This is impossible, hence the claim follows.
\end{proof}

\begin{proof}[Proof of Proposition \ref{cor.transversality in the closure}]
Since condition (Tr-1) holds with respect to $\bx,\by$ we can find $\g\in\G$ so that  $l\tv \rho_0(\g) p$. Let 
$$s=\lim x_i, \quad  t=\lim y_i\in\dg$$
which are distinct by assumption.

We first claim that there exist $h,g\in \G$ such that the following quadruples are cyclically ordered (see Figure \ref{fig:p3.4bis}) 
$$(\bbx,\g\bby,h\g\bby,h\bbx), (\g^{-1} \bbx,\bby,g\g\bby,g\bbx), (\bbx,h\g\bby, \bby,g\bbx).$$ 

	Indeed, observe first that $s\neq \gamma t$: otherwise a diagonal sequence argument allows to write $l=\lim \xi^k_{n_i}(x_i')$ and $\rho_0(\gamma) p=\lim \xi^{d-k}_{n_i}(x_i')$, which contraddicts transversality.
It then follows that,	up to reversing the order of the circle there are three possible configuration for the 4-tuple $\{\bbx,\g\bby,\g^{-1} \bbx, \bby\}$ depicted in Figure \ref{fig:p3.4bis}. In each case we can choose elements $g,h$ with fixed point in the  intervals indicated in the figure, and sufficiently big translation distance:
\begin{center}
	\begin{figure}[h]
		\begin{tikzpicture}[scale=.8]
			\draw (1.5,0) circle [radius =1.5];
			\filldraw (0,0) circle [radius=1pt] node [left] {$\bbx$};
			\filldraw (0.1,0.55) circle [radius=1pt] node [left] {$\g^{-1}\bbx$};
			\filldraw (0.6,1.2) circle [radius=1pt] node [above ] {$h\bbx$};
			\filldraw (1.5,1.5) circle [radius=1pt] node [above] {$h_+$};
			\filldraw (2.6,1) circle [radius=1pt] node [right] {$h\g\bby$};
			\filldraw (2.9,0.55) circle [radius=1pt] node [right] {$g_-$};
			\filldraw (3,0) circle [radius=1pt] node [right] {$\bby$};
			\filldraw (2.6,-1) circle [radius=1pt] node [right] {$\g\bby$};
			\filldraw (2,-1.43) circle [radius=1pt] node [below] {$g\g\bby$};
			\filldraw (1.5,-1.5) circle [radius=1pt] node [below] {$g_+$};	
			\filldraw (1,-1.43) circle [radius=1pt] node [below] {$g\bbx$};
			\filldraw (0.5,-1.1) circle [radius=1pt] node [left] {$h_{-}$};
		\end{tikzpicture}
			\begin{tikzpicture}[scale=.8]
	\draw (1.5,0) circle [radius =1.5];				
	\filldraw (0,0) circle [radius=1pt] node [left] {$\bbx$};
	\filldraw (0.1,0.55) circle [radius=1pt] node [left] {$h\bbx$};
	\filldraw (0.6,1.2) circle [radius=1pt] node [above ] {$h_+$};
	\filldraw (1.5,1.5) circle [radius=1pt] node [above] {$h\g\bby$};
	\filldraw (2.6,1) circle [radius=1pt] node [right] {$\bby$};
	\filldraw (2.9,0.55) circle [radius=1pt] node [right] {$\g^{-1}\bbx$};
	\filldraw (3,0) circle [radius=1pt] node [right] {$\g\bby$};
	\filldraw (2.6,-1) circle [radius=1pt] node [right] {$g\bbx$};
	\filldraw (2,-1.43) circle [radius=1pt] node [right] {$g\g\bby$};
	\filldraw (1.5,-1.5) circle [radius=1pt] node [below] {$g_+$};	
	\filldraw (1,-1.43) circle [radius=1pt] node [below] {$g_{-}$};
	\filldraw (0.5,-1.1) circle [radius=1pt] node [left] {$h_-$};
\end{tikzpicture}
		\begin{tikzpicture}[scale=.8]
				\draw (1.5,0) circle [radius =1.5];
			\filldraw (0,0) circle [radius=1pt] node [left] {$\bbx$};
\filldraw (0.1,0.55) circle [radius=1pt] node [left] {$\g^{-1}\bbx$};
\filldraw (0.6,1.2) circle [radius=1pt] node [above ] {$g\bbx$};
\filldraw (1.5,1.5) circle [radius=1pt] node [above] {$g_+$};
\filldraw (2.6,1) circle [radius=1pt] node [right] {$g\g\bby$};
\filldraw (2.9,0.55) circle [radius=1pt] node [right] {$h_-$};
\filldraw (3,0) circle [radius=1pt] node [right] {$\g\bby$};
\filldraw (2.6,-1) circle [radius=1pt] node [right] {$\bby$};
\filldraw (2,-1.43) circle [radius=1pt] node [below] {$h\g\bby$};
\filldraw (1.5,-1.5) circle [radius=1pt] node [below] {$h_+$};	
\filldraw (1,-1.43) circle [radius=1pt] node [below] {$h\bbx$};
\filldraw (0.5,-1.1) circle [radius=1pt] node [left] {$g_{-}$};
\end{tikzpicture}
		\caption{The three possible configuration of points in the proof of Proposition \ref{prop.limits of spr: transverse orbits}. In all cases the 4-tuples $(\bbx,\g\bby,h\g\bby,h\bbx), (\g^{-1} \bbx, \bby,g\g\bby,g\bbx), (\bbx,h\g\bby, \bby,g\bbx)$ are cyclically ordered }\label{fig:p3.4bis}
	\end{figure}
\end{center}
Let then $(n_i)$ be the subsequence given by Lemma \ref{lem.diagonal-sequence} so that 
$$l=\lim_{i\to\infty}\xi^k_{\rho_{n_i}}(x_i), \quad p=\lim_{i\to\infty}\xi^{d-k}_{\rho_{n_i}}(y_i).$$
Applying  Lemma \ref{lem opposition for limits of pos ratioed} for the first two quadruples while recalling that $\lim_i x_i=s, \lim y_i=t$ gives the transversalities 
$$l\tv \rho_0(h\g)p, \quad \rho_0(g)l\tv p.$$ 
We can then apply the lemma to the third quadruple to derive $l\tv p$.
\end{proof}

We will say that the  maps $\xi^k_\bx:\G \bx\to \Gr_k(E)$ and $\xi^{d-k}_\by:\G \by\to \Gr_{d-k}(E)$ defined  on the $\G$--orbits of the  points $\bx,\by\in \dg$ are \emph{positively ratioed} if the maps are transverse (namely for every $x_0\in \G \bx$ and $y_0\in\G\by$ with $x_0\neq y_0$, $\xi^k_\bx(x_0)\tv \xi^{d-k}_\by(y_0)$) and  
$$\cro_{k}(\xi^k_\bx(u),\xi^{d-k}_\by(z),\xi^{d-k}_\by(w),\xi^k_\bx(v))\geq 1$$
for $u,v\in \G \bx$, $z,w\in \G \by$ and $u,z,w,v$ in that cyclic order in $\dg$.

Combining Proposition \ref{cor.transversality in the closure}  and Lemma \ref{lem opposition for limits of pos ratioed} we also obtain
\begin{prop}\label{prop.limits of spr: transverse orbits}
	If the sequence $\{\rho_n\}$ satisfies condition (Tr) with respect to $\bx,\by\in \dg$, then 
	%$$\xi^k_\bx(z)\tv \xi_\by^{d-k}(w), \quad \forall (z,w) \in \G \bx \times \G \by,\quad  z\neq w,$$
	%and 
	the maps $\xi^k_\bx$, $\xi^{d-k}_\by$ are positively ratioed.
\end{prop}

\subsection{Boundary map for the limit representation}\label{s.bdry}
We show  that condition (Tr) guarantees the existence of continuous transverse boundary maps. Throughout the section we assume that 	the sequence $\{\r_n\}$  satisfies {condition $\rm{(Tr)}$ with respect to $\bx,\by\in \dg$}, so that, in particular $\xi^k_{\r_n}(\bx)$ (resp. $\xi^{d-k}_{\r_n}(\by)$) converge and  let us define equivariant boundary maps $\xi^k_{\bx}$ (resp. $\xi^{d-k}_\by$) on their orbits.
First we prove that left and right continuous extensions of $\xi_{\bx}^k$ and $ \xi_{\by}^{d-k}$ exist: %thanks to  (Tr-2) and the next lemma. We then use condition (Tr-2) again to show that those maps are  continuous.

\begin{lem}\label{lem.cro left limit exists}
If $x_n,x'_n\in \G \bx$ converge to $x_0$ and $(x_n,x'_n,x_j,x'_j,x_0)$ are in cyclic order for any $n<j$, then 
$$\cro_{k}(\xi_{\bx}^k(x_n),\xi_{\by}^{d-k}(y_1),\xi_{\by}^{d-k}(y_2),\xi_{\bx}^k(x'_n))\to 1$$ 
for any distinct $y_1,y_2\in\G \by \backslash\{x_0\}$.
\end{lem}
\begin{proof}
After possibly switching the roles of $y_1$ and $y_2$ we can find $n_0$ such that  the points $y_1,x_n,x'_n,x_j,x'_j,y_2$ are in  cyclic order in $\dg$ if $n_0\leq n<j$. 
Choose $z\in \G \bx \cap (x_0,y_2)_{y_1} $ (see Figure \ref{f.3.7}).

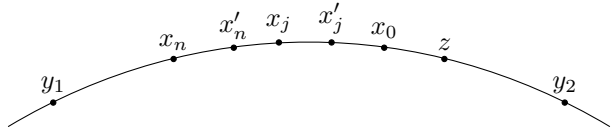
\begin{center}
\begin{figure}[h]
\begin{tikzpicture}
\draw (-4,0) .. controls (-1.5,1.5) and (1.5,1.5) .. (4,0);
\filldraw (-3.4,0.32) circle [radius=1pt] node [above] {$y_1$};
\filldraw (3.4,0.32) circle [radius=1pt] node [above] {$y_2$};
\filldraw (1,1.05) circle [radius=1pt] node [above] {$x_0$};
\filldraw (0.3,1.12) circle [radius=1pt] node [above ] {$x_j'$};
\filldraw (-.4,1.12) circle [radius=1pt] node [above ] {$x_j$};
\filldraw (-1,1.05) circle [radius=1pt] node [above ] {$x_n'$};
\filldraw (-1.8,0.9) circle [radius=1pt] node [above ] {$x_n$};
\filldraw (1.8,0.9) circle [radius=1pt] node [above ] {$z$};
\end{tikzpicture}
\caption{The configuration of points in the proof of Lemma \ref{lem.cro left limit exists}}\label{f.3.7}
\end{figure}
\end{center}
To simplify notation set $Y_i:=\xi_{\by}^{d-k}(y_i)$. % and $Y_2:=\xi^{d-k}(y_2)$. Then 
Iteratively applying the cocycle identity of the cross ratio we get  
\begin{align*}
&\cro_{k}\left(\xi_{\bx}^k(x_{n_0}),Y_1,Y_2,\xi_{\bx}^k(z)\right)\\
=&\cro_k\left(\xi_{\bx}^k(x_{j+1}),Y_1,Y_2,\xi_{\bx}^k(z)\right)\prod_{i=n_0}^j \cro_k\left(\xi_{\bx}^k(x_i),Y_1,Y_2,\xi_{\bx}^k(x'_i)\right)\cro_k\left(\xi_{\bx}^k(x'_i),Y_1,Y_2,\xi_{\bx}^k(x_{i+1})\right)
\end{align*}
for all $j>n_0$. Each of the factors in the product is at least $1$, because the boundary maps are positively ratioed and the  quadruples are cyclically ordered. Since the left hand side is finite,  the terms in the product necessarily converge to 1.
\end{proof}

We fix some (auxiliary) orientation on $S^1\simeq \dg$, so that we can talk about convergence from the left (denoted $x^{-}\to x_0$) or from the right (denoted $x^{+}\to x_0$). 

\begin{lem}\label{lem.left limit exists}
Left-limits $\lim_{x^{-}\to x_0} \xi^k_{\bx}(x^{-})$ and right-limits $\lim_{x^{+}\to x_0} \xi^k_{\bx}(x^{+})$ 
exist and the same holds for $\xi^{d-k}_{\by}$. They induce $\r_0$--equivariant maps 
$$\xi^k_{-},\xi^k_{+}:\dg\to\Gr_k(E)\quad \xi^{d-k}_{-},\xi^{d-k}_{+}:\dg\to\Gr_{d-k}(E)$$
with $\xi^k_{-},\xi^{d-k}_{-}$ left-continuous, $\xi^k_{+},\xi^{d-k}_{+}$   right-continuous. Each of the boundary maps $\xi^k_{-},\xi^k_{+},\xi_{\bx}^k$ is transverse to each of $\xi^{d-k}_{-},\xi^{d-k}_{+}$, $\xi^{d-k}_{\by}$.
\end{lem}

\begin{proof} 
Condition (Tr-2), combined with Lemma \ref{lem.cro left limit exists}, guarantees that left and right limits $\xi^k_{\pm}:\dg\to  \Gr_k(E)$ exist and are unique: Indeed if, by contradiction, two different accumulation points $l,l'$ of sequences converging to $x_0$ from the same side existed, we could pass to a subsequence and find two nested sequences $\{x_n\},\{x'_n\}\subset \G \bx$ so that $\xi^k_{\bx}(x_n)$ and $\xi^k_{\bx}(x'_n)$ converge to $l$ and $l'$, respectively, and find $z,z'\in \G \by\backslash \{x_0\}$ such that
$$\cro_k(l,\xi^{d-k}_{\by}(z),\xi^{d-k}_{\by}(z'),l')\neq 1.$$ 
This would contradict Lemma \ref{lem.cro left limit exists} as the lemma yields
$$\cro_k(l,\xi^{d-k}_{\by}(z),\xi^{d-k}_{\by}(z'),l')=\lim_{i\in \N} \cro_k(\xi^k_{\bx}(x_{n}),\xi^{d-k}_{\by}(z),\xi^{d-k}_{\by}(z'),\xi^k_{\bx}(x'_{n}))=1.$$

The same argument with the roles of $\xi^k_{\bx}$ and $\xi^{d-k}_{\by}$ exchanged provides left-, respectively right-continuous maps $\xi^{d-k}_{-},\xi^{d-k}_{+}:\dg\to\Gr_{d-k}(E)$.
By construction $\xi^k_{\pm}(\dg)\subset \overline{\xi^k_{\bx}(\G\bx)}$, and $\xi_{\pm}^{d-k}(\dg)\subset \overline{\xi^{d-k}_{\by}(\G\by)}$. Proposition \ref{cor.transversality in the closure} ensures the transversality of the boundary maps.
\end{proof}

We now prove the main result of this section.
\begin{thm}\label{thm.limits of 1-pos ratio}
	If condition $\rm{(Tr)}$ holds for some $\bx,\by\in\dg$,
	there are $\r_0$--equivariant, continuous, transverse, $k$--positively ratioed boundary maps 
	$$\xi^k_{\r_0}:\dg\to \Gr_k(E),\quad \xi^{d-k}_{\r_0}:\dg\to \Gr_{d-k}(E).$$
\end{thm}

\begin{proof}  We first show that the right and left limits $\xi^k_{-}, \xi^k_{+}$ agree, and the same holds for $\xi^{d-k}_{-}, \xi^{d-k}_{+}$, so that  $\xi^k_{-},\xi^{d-k}_{-}$ is a pair of $\r_0$--equivariant continuous transverse boundary maps. 
Assume by contradiction  that there exists $w\in \dg$ such that $\xi^k_-(w)\neq \xi^k_+(w)$. 
Since $\G$ is hyperbolic, $w$ is a conical limit point, i.e. there is a sequence $\{\g_n\}\subset \G$ and distinct $w',z'\in \dg$ such that $\g_n w\to w'$ and $\g_n z\to z'$ for every $z\in\dg \backslash\{w\}$.
Choose  $z_0\in \G\by \backslash\{w\}$. After possibly passing to a subsequence of $\{\g_n\}$ we can assume that $\g_n z_0$ converges from one side to $z'$; assume without loss of generality that it is from the right. Denote by $(w,z_0)^r$ the connected component of $\dg\backslash \{w,z_0\}$ such that $z'\notin \g_n(w,z_0)^r$ for all big enough $n$.
Since the convergence is conical, the action is by {{orientation preserving}} homeomorphisms and $\g_n z_0\to z'$ from the right, we deduce that  $\g_n z\to z'$ from the the right for every $z\in (w,z_0)^r$.  Lemma \ref{lem.left limit exists} implies $\xi^{d-k}_{\by}(\g_n z)\to \xi^{d-k}_{+}(z')$ for all $z\in (w,z_0)^r\cap \G\cdot \by$.

Since  $\Gr_k(E)$ is compact, we can pass to a subsequence of $\{\g_n\}$, still denoted  $\{\g_n\}$, such that $\xi^k_{-}(\g_n w)$ and $\xi^k_{+}(\g_n w)$  converge. We denote the respective limits by $l_-,l_+\in\Gr_k(E)$

By condition (Tr-2) we find $y_1,y_2\in (w,z_0)^r\cap \G \by$ such that
$$\cro_k(\xi^k_-(w),\xi^{d-k}_{\by}(y_1),\xi^{d-k}_{\by}(y_2),\xi^k_+(w))\neq 1.$$ 
Since $l_{\pm}\in \overline{\xi^k_{\pm}(\dg)} \subset \overline{\xi^k_{\bx}(\G\bx)}$, Proposition \ref{cor.transversality in the closure} guarantees that $l_{\pm}\tv \xi^{d-k}_{+}(z')$.
The continuity of the cross ratio and the invariance by $\PGL(E)$ give that
\begin{align*}
&\cro_{k} (\xi^k_{-}(w),\xi^{d-k}_{\by}(y_1) ,\xi^{d-k}_{\by}(y_2),\xi^k_{+}(w))\\
=&\cro_{k} (\xi^k_{-}(\g_n w),\xi^{d-k}_{\by}(\g_n y_1) ,\xi^{d-k}_{\by}(\g_n y_2),\xi^k_{+}(\g_n w))\\
 =&\cro_{k} (l_-,\xi^{d-k}_{+}(z') ,\xi^{d-k}_{+}(z'),l_+)=1,
\end{align*} 
where in the second to last equality we let $n\to \infty$. This is a contradiction, thus $\xi^k_{-}=\xi^k_{+}$. The $(d-k)$-case is entirely analogue.

The pair $(\xi^k_-,\xi^{d-k}_-)$ is  $k$--positively ratioed by continuity of $\cro_{k}$ and the definitions of $\xi^k_{-},\xi_{-}^{d-k}$. Indeed, for any cyclically ordered $(x,y,y',x')$ in $\dg$,  we find sequences $\{x_n\},\{x_n'\}\subset \G \bx$ and $\{y_n\},\{y_n'\}\subset \G \by$ converging to the respective points from the left by minimality of $\G$--action on $\dg$. Since for $n$ big enough $(x_n,y_n,y_n',x_n')$ is cyclically ordered,
\[\cro_{k}(\xi^k_{\bx}(x_n),\xi_{\by}^{d-k}(y_n),\xi_{\by}^{d-k}(y_n'),\xi^k_{\bx}(x'_n))\geq 1.\qedhere \]
\end{proof}
\begin{remark}\label{r.limits of 1-pos ratio}
The boundary maps $\xi^k_{\rho_0}$, $\xi^{d-k}_{\rho_0}$ constructed in Theorem \ref{thm.limits of 1-pos ratio} might depend on the choice of points $\bx$, $\by$. Observe that, by construction $\xi^k_{\rho_0}=\xi^k_\bx$   (resp. $\xi^{d-k}_{\rho_0}=\xi^{d-k}_\bx$) on their common domain of definition, namely the $\G$-orbit $\G\bx$ (resp. $\G\by$).
\end{remark}	
\begin{remark}
There is a natural $\R$-valued cross ratio on a subset of $\quotient{G}{P}\times \quotient{G}{P^-}\times \quotient{G}{P^-}\times \quotient{G}{P}$ for any (semi-)simple Lie group $G$ and  any maximal parabolic $P<G$ with  opposite parabolic $P^-$ (see for instance \cite{Beyrer}). Condition (Tr) can be defined analogously, and Theorem \ref{thm.limits of 1-pos ratio} can be proven in the same way in such more general setting. However observe that not all pairs $(G,P)$ admit Zariski dense $P$-positively ratioed representations (e.g. if $G=\PSL(\mathbb{C}^2)$ this is not the case).%\mJ{Add that it seems likely that the root multiplicity of the simple root defining the conjugacy class of $P$ should be 1??}
\end{remark}

\section{Limits of positively ratioed representations}\label{s.pos-ratio}
In this section we discuss several  instances in which Theorem \ref{thm.limits of 1-pos ratio} can be used to derive that the limit is (almost) Anosov.
As always we consider sequences $(\rho_n)$ of $k$--positively ratioed representations  converging to a representation $\rho_0$.

\subsection{Zariski dense limits} Zariski density guarantees that the hypotheses of Theorem \ref{thm.limits of 1-pos ratio} are satisfied:
\begin{prop}\label{prop.Tr Zariski dense}
	If $\r_0$ is Zariski dense, then for any $\bx$, $\by$ in $\dg$ there exists a subsequence of $(\r_n)$ that satisfies condition $\rm{(Tr)}$ with respect to  $\bx$, $\by$.
\end{prop}

\begin{proof}
	Choose $\bx,\by\in \dg$. Since the Grassmannian is compact, we can pass to a subsequence to guarantee that $\lim \xi_{\r_n}^k(\bx)$ and $\lim \xi_{\r_n}^{d-k}(\by)$ exist. Then condition (Tr-1) is an immediate consequence of Zariski density: if there existed $l\in\Gr_k(E)$, $p\in\Gr_{d-k}(E)$ such that for every $\g\in\G$, $\rho_0(\g)p\ntv l$, the representation $\rho$ would not be Zariski dense. Since the cross ratio is algebraic, for every pair of distinct $l, l'\in \Gr_k(E)$, and every $z\in U$  we can find $\g\in \G$ such that 		
	$$\cro_k(l,\xi^{d-k}_{\by}(z),\xi^{d-k}_{\by}(\g \by),l')\neq 1.$$ 
	The same argument as in Proposition \ref{prop.labourie-irreducible} shows that we can also find an element $\g$ such that the orbit point $z'=\g\by$ lies in a given subset $U\subset \dg$. 
\end{proof}
Since for Zariski dense representations admitting equivariant, continuous transverse boundary maps already implies that the representation is Anosov (Proposition \ref{prop.GW-irreducibility}) we deduce from Theorem \ref{thm.limits of 1-pos ratio}:
\begin{cor}\label{cor:Zd}
	A Zariski dense limit $\rho_0$ of $k$-positively ratioed representations is $k$-positively ratioed. 
\end{cor}	
\subsection{Irreducible limits}
If the sequence is $1$--positively ratioed, in Corollary \ref{cor:Zd} Zariski density can be relaxed to irreducibility (Proposition \ref{prop.irred limits of 1-pos ratioed}). We verify that the limit satisfies condition (Tr) under weaker assumptions, which will be useful in other contexts as well. The following proposition includes irreducible representations, for which we let $E_2$ be the trivial subspace.

\begin{prop}\label{pro.condition L irreducible}
	If $\{\r_n\}_{n\in \N}$ are 1--positively ratioed, $\r_0$ preserves a splitting $E=E_1\oplus E_2$ with $E_1$  irreducible, and there exist $\bx,\by\in \dg$ such that the limit $\lim \xi^1_{\r_n}(\bx)$ exists and  is contained   in $E_1$ and the  limit $\lim \xi^{d-1}_{\r_n}(\by)$ exists and contains $E_2$, %< \lim \xi^{d-1}_{\r_n}(\by)$,\footnote{In particular we assume that the limits exist.} 
	then $\{\r_n\}$ satisfies condition $\rm{(Tr)}$ with respect to $\bx,\by\in \dg$.
\end{prop}

\begin{proof}
	Since $E_1$ is irreducible, for each $l\in \P(E_1)$ and each $p\in \Gr_{d-1}(E)$ with $E_2\subset p$ we find $\g$ such that $\r_0(\g)l\tv  p$. As $\overline{\xi^1_{\bx} (\G\bx)}\subset \P(E_1)$ and $E_2\subset p$ for all $p\in \overline{\xi^{d-1}_{\by} (\G\by)}$, (Tr-1) follows.
	
	Given $x_0\in\dg$ and an open set $U\subset\dg\setminus \{x_0\}$, we can find  $z_1,\ldots, z_r\in \G \by\cap U$ such that 
	\begin{align}\label{eq.condition Tr for 1-pos}
		\bigcap_{i=1,\ldots,r} \xi^{d-1}_{\by}(z_i)= E_2.
	\end{align}
	Indeed there are $r=\dim E_1$ points $w_i\in \G\by$ satisfying \eqref{eq.condition Tr for 1-pos} because $E_1$ is irreducible and $E_2<\xi^{d-1}_{\by}(\by)$; for a sufficiently high power of a loxodromic element $\g\in\G$ with attracting point in $U$ the points $z_i:=\g w_i$ have the desired property.
	
	Let now $l, l'$ be as in (Tr-2a). By 
	Proposition \ref{cor.transversality in the closure},  $l,l'$ are  transverse to all of $Z_i:=\xi^{d-1}_{\by}( z_i)$ for $i=1,\ldots,r$. If, by contradiction, $\cro_1(l,Z_i,Z_j,l')=1$ for all $i,j$, then it follows from Remark \ref{rem.cross ratio as projective} that
	$$Z_i\cap (l\oplus l')= Z_j\cap (l\oplus l')\quad \forall
	i,j.$$
	Thus $o:=(Z_1\cap (l\oplus l'))\in Z_j$ for all $j=1,\ldots,r$. Since $o\in E_1$,  this contradicts Equation \eqref{eq.condition Tr for 1-pos}. Condition (Tr-2b) follows reversing roles of $\xi^1_{\bx}$ and $\xi^{d-1}_{\by}$.
\end{proof}

As a result we obtain that if a limit of $1$-positively ratioed representations is irreducible than it is $1$-positively ratioed, and thus in particular Anosov: 
\begin{prop}\label{prop.irred limits of 1-pos ratioed}
	Let $\{\r_n\}_{n\in \N}$ be a sequence of $1$--positively ratioed representations converging to an irreducible representation $\r_0$. Then $\r_0$ is $1$--positively ratioed.
\end{prop}

\begin{proof}
	In this case a subsequence satisfies condition (Tr) by Proposition \ref{pro.condition L irreducible}. Hence $\r_0$ admits continuous transverse positively ratioed boundary maps (Theorem \ref{thm.limits of 1-pos ratio}). Hence $\r_0$ is 1--Anosov and the latter boundary maps are the Anosov boundary maps (Proposition \ref{prop.GW-irreducibility} by Guichard--Wienhard).
\end{proof}

\subsection{Proximal limits, the proof of Theorem \ref{thmINTRO.limits of pos ratio general+proximal}}

The assumptions of Proposition \ref{pro.condition L irreducible} are satisfied if there is an element $\g\in\G$ such that first generalized eigenspace of $\r_0(\g)$ is contained in a single irreducible factor. This will be key in the construction of boundary maps for $k$-proximal limits of $k$-positively ratioed representations.
\begin{cor}\label{cor.weakly Anosov limit generalizing proximality}
	If $\{\r_n\}_{n\in \N}$ are 1--positively ratioed, $\r_0=\eta_1\oplus \eta_2$  
	preserves the splitting $E=E_1\oplus E_2$ with $E_1$ irreducible, and there exists an infinite order element $\g\in\G$ with
	$$\left|\frac{\lambda_1^{\eta_1}}{\lambda_1^{\eta_2}}(\g)\right|>1,$$
 then a subsequence of $\{\r_n\}$ satisfies condition $\rm{(Tr)}$, with respect {{to $\bx=\g_+$, $\by=\g_-$}}.
\end{cor}

\begin{proof}
	The assumptions imply that any limit of the sequence $\{\xi_{\r_n}^1(\g_+)\}$ belongs to $\P(E_1)$ and $E_2$ is contained in any limit of $\{\xi_{\r_n}^{d-1}(\g_-)\}$. It remains to pass  to a subsequence $\{n_i\}_{i\in \N}$ such that both $\xi_{\r_{n_i}}^{1}(\g_+)$ and $\xi_{\r_{n_i}}^{d-1}(\g_-)$  converge, which might be necessary as $\eta_1(\g)$ is not assumed to be 1-proximal. 
\end{proof}
Corollary \ref{cor.weakly Anosov limit generalizing proximality} applies, in particular, if $\r_0(\g)$ is $1$-proximal, namely it contains a 1-proximal element in its image. We will now show that the boundary map provided by Theorem \ref{thm.limits of 1-pos ratio} is, in this case, dynamics preserving at all fixed points of proximal elements, of which Theorem \ref{thmINTRO.limits of pos ratio general+proximal} is a direct consequence.

\begin{prop}\label{prop:kproxlim}
	A $k$--proximal, reductive limit $\rho_0$ of $k$-positively ratioed representations $\rho_n$ admits continuous, transverse equivariant boundary maps that are dynamics preserving at the fixed points  of all $k$--proximal elements.
\end{prop}	
\begin{proof} In order to prove the proposition we need to verify that the sequence satisfies condition (Tr) and that the maps provided by Theorem \ref{thm.limits of 1-pos ratio} are dynamics preserving.
	
	{\bf Case 1: $k=1$.}
	Choose $\g \in\G$ be  such that $\r_0(\g)$ is 1--proximal; we  write $$\r_0=\eta_1\oplus \eta_2:\G\to \PGL(E_1\oplus E_2)$$ 
	where $E_1$ the irreducible factor containing the top eigenspace for $\r_0(\g)$.
	By Corollary \ref{cor.weakly Anosov limit generalizing proximality} the sequence satisfies condition $\rm{(Tr)}$ with respect to $\bx=\g_+$ and $\by=\g_-$. We consider the  equivariant, transverse, continuous, 1-positively ratioed boundary maps 
	$$\xi^1_{\r_0}:\dg\to \P(E),\quad \xi^{d-1}_{\r_0}:\dg\to \Gr_{d-1}(E).$$
	obtained from Theorem \ref{thm.limits of 1-pos ratio}. 
	By construction, $\xi^1_{\r_0}(\dg)\subset\overline{ \xi^1_{\bx}(\G\bx)}\subset \P(E_1)$ and $E_2\subset \xi^{d-1}_{\r_0}(x)$ for any $x\in \dg$, as $\xi^{d-1}_{\r_0}(\dg)\subset\overline{ \xi^{d-1}_{\by}(\G\by)}$ and $E_2\subset  \xi^{d-1}_{\by}(\by)$.
	In particular $\eta_1:\G\to\PGL(E_1)$, admits boundary maps with the same properties, 
	which are the 1--Anosov boundary maps of $\eta_1$ by the irreducibility of $E_1$  (Proposition \ref{prop.GW-irreducibility}).  
	Thus, for every $g\in\G$, $\xi^1_{\r_0}(g_+)$ is the eigenspace of $\lambda_1^{\r_0}(g)$ and $\xi^{d-1}_{\r_0}(g_-)$ is the unique repelling hyperplane of $\r_0(g)$.
	
	It remains to show that for each 1-proximal $\r_0(g)$, the eigenspace $E^{\lambda_1}_{\rho_0(g)}$ is contained in $E_1$. Assume by contradiction that this is not the case for an element $g\in \G$.
	Since the eigenvalues are conjugation invariant, we can assume without loss of generality that $(\g_+,\g_-,g_-,g_+)$ are in that cyclic order in $\dg$. 
	Since $\r_0(g)$ is $1$--proximal, the limit $\lim_n \xi^1_{\r_n}(g_+)$ exists, equals $E^{\lambda_1}_{\rho_0(g)}$ and thus lies in $\P(E_2)$. In the same way since $\tilde{\r}_0(g^{-1})$ is $(d-1)$--proximal, $V^{d-1}=\lim \xi^{d-1}_{\r_n}(g_-)$ exists, is transverse to $\lim \xi^1_{\r_n}(g_+)$ and contains $E_1$.
	The continuity of the cross ratio implies 
	$$\cro_1(\xi^1_{\r_n}(\g_+),\xi^{d-1}_{\r_n}(\g_-),\xi^{d-1}_{\r_n}(g_-),\xi^1_{\r_n}(g_+))\to\cro_1(\xi^1_{\r_0}(\g_+),\xi^{d-1}_{\r_0}(\g_-),V^{d-1},E^{\lambda_1}_{\rho_0(g)}),$$
	but each of the cross ratios on the left is greater than 1 and the cross ratio on the right is defined and equal to 0; a contradiction. 
	
	{\bf Case 2: $k$ general.}
	If the representations $\r_n:\G\to \PGL(E)$ satisfy the assumptions of the proposition for some $k$, then their exterior powers $\wedge^k \r_n:\G\to \PGL(\wedge^k E)$ satisfy the assumptions for $k=1$: indeed being reductive is a property of the Zariski closure, which is not changed through algebraic representations of the group, and the $k$-th exterior power behaves well with respect to the cross ratio and boundary maps  (cfr. Proposition \ref{prop.pos ratio exterior power}).  Moreover by construction the limit maps associated to $\wedge^k\rho_0$ are pointwise limits of the boundary maps associated to $\wedge^k\rho_n$, as a result they are themselves of the form $\wedge^k\xi^k_{\rho_0}$ for some  maps $\xi^{k}_{\r_0}:\dg\to \Gr_k(E)$. It follows from the analysis in Section \ref{s.ext}  that the former maps 
	are equivariant, transverse, continuous, $k$--positively ratioed boundary maps and dynamics preserving if and only if the latter are.
\end{proof}
\begin{remark}
	It follows from the proof that, under the assumption of Proposition \ref{prop:kproxlim},  the boundary map doesn't depend on the choice of $\bx=\g_+$, $\by=\g_-$, and for every $k$-proximal $\r_0(\eta)$,  $\xi^{k}_{\r_n}(\eta_+)\to \xi^{k}_{\r_0}(\eta_+)$. However we do not, in general, know how to guarantee pointwise convergence of the boundary maps at every point $x\in\dg$.
\end{remark}

A powerful way to guarantee that the limit representation is proximal is via collar lemmas (recall from Definition \ref{d.collar} that we consider here collar lemmas comparing roots and weights).

\begin{lem}\label{lem.almost Anosov limitcollar}
	Let $\{\r_n\}_{n\in \N}$ be a sequence of representations satisfying a $k$--collar lemma and converging to $\r_0$. Then for every non-trivial $\g\in \G$, $\r_0(\g)$ is $k$--proximal.
\end{lem}

If we combine this with Theorem \ref{thmINTRO.limits of pos ratio general+proximal} we get:

\begin{cor}\label{cor.collar lemma almost Anosov}
	Let $\{\r_n\}_{n\in \N}$ be a sequence of $k$--positively ratioed representations satisfying a $k$--collar lemma and converging to a reductive representation $\r_0$. Then $\r_0$ is $k$--almost Anosov.
\end{cor}

\subsection{$(k-1)$-Anosov limits}
Assuming further Anosovness in the limit ensures that contition (Tr) is satisfied.

\begin{prop}\label{prop.condition Tr strongly irred}
	If the limit $\r_0$  of a sequence of $k$-positively ratioed representations is strongly irreducible  and $(k-1)$--Anosov, then $\rho_0$ 
	admits equivariant, continuous, transverse, $k$-positively ratioed boundary maps.
\end{prop}

\begin{proof}
	Observe that, since being Anosov is an open property, for every $n$ big enough $\rho_n$ is  $(k-1)$-Anosov.
	We will check that for every $\bx,\by\in\dg$ fixed there is a  subsequence of $\{\r_n\}_{n\in\N}$ that satisfies condition (Tr) with respect to $\bx,\by$. This is sufficient in view of Theorem \ref{thm.limits of 1-pos ratio}.
	
	Fix  $\bx,\by\in\dg$. Up to passing to a subsequence, we can assume that $\xi^{k}_{\r_n}(\bx),\xi^{d-k}_{\r_n}(\by)$ converge. 
	Condition (Tr-1) is a consequence of strong irreducibility of $\r_0$: for every $l\in \Gr_k(E)$ and every $p\in\Gr_{d-k}(E)$ there is $\g\in \G$ such that $l\tv \r_0(\g) p$ (\cite[Proposition 10.3]{Labourie-IM}). This gives condition (Tr-1).
	
	Let now $l,l'\in\Gr_k(E)$ be as in (Tr-2a).	
	Proposition \ref{cor.transversality in the closure} implies that $l,l'\tv \xi^{d-k}_{\by}(s)$ for all $s\in \G \by\backslash\{x_0\}$.
	As $\r_n,\r_0$ are $(k-1)$--Anosov and the boundary maps vary continuously with the representation, Lemma \ref{lem.diagonal-sequence} implies that 
	$$\xi^{k-1}_{\r_0}(x_0)\subset l,l'.$$
	
	Denote by $L^{k+1}$ the $(k+1)$--plane spanned by $l,l'$, by $L^{k-1}:=\xi^{k-1}_{\r_0}(x_0)$ and by $L=\quotient{L^{k+1}}{L^{k-1}}$.
	Then for all $w,z\in\G \by\backslash\{x_0\}$
	$$\cro_k(l_,\xi^{d-k}_{\by}(w),\xi^{d-k}_{\by}(z),l') = \mathrm{pcr} ([l],[\xi^{d-k}_{\by}(w)\cap L^{k+1}],[ \xi^{d-k}_{\by}(z) \cap L^{k+1}],[l']),$$
	where $[\cdot]$ denotes the quotient in $\P(L)$ and $\mathrm{pcr}$ the usual projective cross ratio on $\R\P^1\simeq \P(L)$ (Proposition \ref{prop.projection of cross ratio}).
	Let $U\subset \dg$ be an open subset. Possibly shrinking $U$, we can assume that $x_0\notin U$. If, by contradiction, for all $w,z\in U\cap \G\by$ %we would have 
	$$\mathrm{pcr} ([l],[\xi^{d-k}_{\by}(w)\cap L^{k+1}],[ \xi^{d-k}_{\by}(z) \cap L^{k+1}],[l'])=1,$$
	then for all $w\in U\cap \G\by$
	$$\xi^{d-k}_{\by}(w)\cap (o\oplus L^{k-1})\neq \{0\}$$ 
	where $o:=\xi^{d-k}_{\by}( y_0)\cap L^{k+1}$ for some $y_0\in U\cap \G\by$. This  contradicts strong irreducibility (Proposition \ref{prop.labourie-irreducible}), and shows that the sequence satisfies condition (Tr-2) as well.
\end{proof}

Our next goal is to improve on the result of Proposition \ref{prop.condition Tr strongly irred} and show that in this case the boundary maps are dynamics preserving and the representation is $k$-Anosov. We begin with a generalization of  Proposition \ref{prop.GW-irreducibility} (Proposition \ref{lem.dynamics preserving for strongly irred}) that holds for general hyperbolic groups.

\begin{prop}\label{lem.dynamics preserving for strongly irred}
	Let $\G$ be a Gromov hyperbolic group, $\r:\G\to\PGL(E)$ be $(k-1)$--almost Anosov, strongly irreducible and admit a pair of equivariant, transverse continuous boundary maps 
	$$\xi^k:\dg\to \Gr_k(E),\quad \xi^{d-k}:\dg\to \Gr_{d-k}(E).$$
Then $\r$ is $k$--Anosov and $\xi^k,\xi^{d-k}$ are the Anosov boundary maps.
\end{prop}

\begin{proof}
We first show 
that $\xi^k,\xi^{d-k}$ are dynamics preserving. For any non-trivial $\g\in \G$, let $E_{\rho(\g)}^{<\lambda_k}$ be the sum of all generalized eigenspaces relative to the eigenvalues of absolute value strictly smaller than the $k$-th eigenvalues of $\rho(\gamma)$ (cfr. Section \ref{s.notation}) - of course $\dim\left(E_{\rho(\g)}^{<\lambda_k}\right)\leq d-k$.

	Since $\r$ is strongly irreducible, we find by Proposition \ref{prop.labourie-irreducible} a point $x\in \dg\backslash\{\g_{\pm}\}$ such that
	$$\xi^k(x)\cap E_{\rho(\g)}^{<\lambda_k}=\{0\}.$$ 
	Since $\xi^k$ is continuous, we deduce that $\xi^k(\g_+)\subset E_{\rho(\g)}^{\geq\lambda_k}$:  any limit of $\r(\g)^n \xi^k(x)=\xi^k(\g^n x)$ is contained in $E_{\rho(\g)}^{\geq\lambda_k}$. In order to conclude the proof we need to verify that  $\dim E_{\rho(\g)}^{\geq\lambda_k}=k$, or equivalently that $\rho(\g)$ has a gap of index $k$, namely $|\lambda_k|>|\lambda_{k+1}|$.
	
	Assume by contradiction that this is not the case. We consider the real Jordan decomposition of a lift $\tilde{\rho}(\g)\in\GL(E)$; since the splitting $\xi^{k}(\gamma_+)\oplus\xi^{d-k}(\gamma_-)$ is invariant we can find a non-trivial invariant subspace $V_{\lambda_{k+1}}<\xi^{d-k}(\gamma_-)$ relative to an eigenvalue $\lambda_{k+1}$ with $|\lambda_k|=|\lambda_{k+1}|$. Let $W<\xi^{d-k}(\gamma_-)$ be a $\rho(\g)$-invariant complement of $V_{\lambda_{k+1}}$.  Since for any $x\in\dg\backslash\{ \g_{\pm}\}$, $\rho(\g)^n\xi^k(x)\to \xi^k(\gamma_+)$,  $\xi^k(x)$ cannot have any component in $V_{\lambda_{k+1}}$, 
	and thus is contained in $\xi^k(\gamma_+)\oplus W$. This in particular implies that for every $x$, $\xi^k(x)$ meets the subspace $W\oplus\left( \xi^k(\gamma_+)\cap \xi^{d-k+1}(\gamma_-)\right)$ for dimensional reasons, since  $W\oplus\left( \xi^k(\gamma_+)\cap \xi^{d-k+1}(\gamma_-)\right)$ has codimension $(k-1)$ in $\xi^k(\gamma_+)\oplus W$.
In turn this contradicts strong irreducibility (Proposition \ref{prop.labourie-irreducible}) as long as we can show that the dimension of the latter subspace is at most $ d-k$. 

In order to prove the last claim we will verify that the boundary maps are compatible, namely that, for our fixed $\g$, the containment $\xi^{k-1}(\g_+)\subset\xi^{k}(\g_+)$ holds. Indeed, since $|\lambda_{k-1}|>|\lambda_k|$, the vector space $E_{\rho(\g)}^{<\lambda_{k-1}}$ has dimension $d-k+1$. Strong irreducibility guarantees the existence of a point $x$ such that 	$$\dim(\xi^k(x)\cap E_{\rho(\g)}^{<\lambda_{k-1}})=1.$$ 
	Then $\xi^k(x)$ contains a $k-1$ dimensional subspace $V_x$ transverse to $E_{\rho(\g)}^{<\lambda_{k-1}}$. Since $\r(\g^n)V_x\to \xi^{k-1}(\g_+)$, the claim follows. This concludes the proof that the maps $\xi^k,\xi^{d-k}$ are dynamics preserving.

\medskip

Using 	Definition \ref{defn.Anosov}(3), in order to conclude the proof  it is enough to show that the gaps of $\rho$ grow, equivalently it is enough to show that the gaps of the $1$--almost Anosov (Proposition \ref{prop.Anosov for exterior power}) representation $\wedge^k \r$  grow.
{ Since $\r$ is strongly irreducible, it is reductive and thus the same holds for $\wedge^k \r$ (see, for instance \cite[Theorem 22.138]{Milne}).} Thus $\wedge^k \r$ splits as 
$$\wedge^k \r=\eta_1\oplus \eta_2:\G\to \PGL(V_1\oplus V_2)$$ 
with $\eta_1$ is irreducible. We further deduce from Proposition \ref{prop.Anosov for exterior power} and Lemma \ref{lem.splitting boundary map for reducible reps} that $\xi^1_{\wedge^k \r}(\dg)\subset \P(V_1)$. 
The 
boundary maps of $\wedge^k\r$ provide equivariant continuous transverse boundary maps for $\eta_1$. Since $\eta_1$ is irreducible, it is  Anosov (Proposition \ref{prop.GW-irreducibility}). Thus it suffices to show that for every $\g\in\G$, $|\lambda_2^{\eta_1}(\g)|\geq |\lambda_1^{\eta_2}(\g)|$.

Assume by contradiction that  there is $\g\in\G$ such that $\lambda_2^{\eta_1}(\g)< \lambda_1^{\eta_2}(\g)$. In particular $E_{\wedge^k\r(\g)}^{\lambda_2}\subset V_2$ and thus no $p\in E_{\wedge^k\r(\g)}^{\lambda_2}$ can be transverse to $\xi^{d'-1}_{\wedge^k \r}(x)$ for any $x\in \dg$  since $V_2\subset \xi^{d'-1}_{\wedge^k \r}(x)$ (cfr. Lemma \ref{lem.splitting boundary map for reducible reps}), here we denote by $d'$ the dimension of $\wedge^k E$. On the other hand, choose a non-zero vector  $w$  in $ E^{\lambda_{k+1}}_{\r(\g)}$, and let $W:=\langle \g_+^{k-1}, w\rangle\in\Gr_k(E)$. Then $\wedge^{k-1} \g_+^{k-1}\wedge w=\wedge^kW$ belongs to $E_{\wedge^k\r(\g)}^{\lambda_2}$. Using strong irreducibility we find $x\in \dg$ such that $x^{d-k}\tv W$ (Proposition \ref{prop.labourie-irreducible}), which can be reformulated, using exterior powers, as  $\wedge^k W\tv \xi^{d'-1}_{\wedge^k \r}(x)$  (cfr. Section \ref{s.ext}) - a contradiction.

\end{proof}

The following corollary is a direct consequence of Propositions \ref{prop.condition Tr strongly irred} and \ref{lem.dynamics preserving for strongly irred}.
\begin{cor}\label{cor.strongirredk-1Anosov}
Let $\{\r_n\}_{n\in \N}$ be $k$--positively ratioed and converge to a strongly irreducible $(k-1)$--Anosov representation $\r_0$. Then $\r_0$ is $k$--positively ratioed.
\end{cor}

%%%%%%%%%%%
\section{Limits of representations satisfying $H_k$}\label{s.gaps}
In this section we study limits of representations satisfying property $H^*_k$. As opposed to the results of Section \ref{s.pos-ratio}, we do not assume anything on the $k$--boundary map but only on the $k-1$ and $k+1$ boundary maps: If a limit of such representations is $\{k-1,k+1\}$--almost Anosov and the boundary maps converge pointwise, then $\rho$ is $k$--Anosov and its boundary maps satisfy property $H_k$. This will be crucially used to show that $k$-positive representations form connected components of $k$--Anosov representations.
 
  We deal with the case $k=1$ first (Proposition \ref{prop.weakly H_1 under weakly 2-Anosov}), and deduce from this the general case (Corollary \ref{cor.closdeness Hk under Anosov}). 

\begin{notation}
In the section $\{\r_n\}$ denotes a sequence of representations $\{\r_n:\G\to\PGL(E)\}_{n\in \N}$ converging to $\r_0\in \Hom(\G,\PGL(E))$.
\end{notation}

As in Section \ref{s.boundary map}, we denote by   $\xi^1_{\bx}:\G \bx\to \P(E),\quad \xi^1_{\bx}(s):=\lim \xi_{\r_n}^1(s)$ the equivariant map, which is well defined as soon as $\lim \xi_{\r_n}^1(\bx)$ exists. A standard result from analysis about convergence of $C^1$ maps guarantees that the limiting maps are compatible, and the limit map is a pointwise limit everywhere:
\begin{lem}\label{lem.C^1 map in the limit}
Let $\rho_n$ be $\{1,2\}$-Anosov representations converging to a representation $\rho_0$. Assume in addition that
\begin{enumerate}
\item $\r_n$ satisfies  property $H_1$ for all $n$;
\item $\{\xi^2_{\r_n}\}$ converges to an equivariant continuous boundary map $\xi^2_{\r_0}$;
\item  $\xi^1_{\bx}$ has a continuous extension $\xi^1_{\r_0}:\dg\to\P(E)$, for some $\bx\in\dg$.% to all of $\dg$;
\end{enumerate}
Then $\{\xi^1_{\r_n}\}_{n\in\N}$ converge to $\xi^1_{\r_0}$ everywhere,  $\xi^1_{\r_0}$ has $C^1$--image and is  tangent to $\xi^2_{\r_0}$.
\end{lem}

\begin{proof}
Choose $s\in\G \bx$ and an   affine chart $A\subset \P(E)$ around  $\xi^1_{\bx}(s)$. For every $n$ we parametrize the image of the restriction of $\xi^1_{\r_n}$ to a small neighbourhood of $s\in\dg$ by a  $C^1$ map $f_n:(-\epsilon,\epsilon)\to A$ with constant norm of the derivative and such that $f_n(0)=\xi_{\r_n}^1 (s)$ (Proposition \ref{prop.1-hyperconvex for property Hk}). Since $\{\xi^2_{\r_n}\}$ converges to a continuous map, the derivatives $\{f'_n\}$ converge to a continuous map $f'_0$. %A standard fact of analysis\marginpar{\B Add a reference?} guarantees that 
Thus the maps $\{f_n\}$ converge to $C^1$ function $f_0$ with derivative $f'_0$. By (3) $\xi^1_{\bx}$ has a continuous extension $\xi^1_{\r_0}$. Since  the extension is unique, the map $f_0$ parametrizes $\xi^1_{\r_0}$ locally around $\xi^1_{\r_0}(s)=f_0(0)$. It follows that,  locally around $s$, $\xi^1_{\r_0}$ has $C^1$ image and tangents  given by $\xi^2_{\r_0}$. As $\G \bx\subset \dg$ is dense, the claim follows.
\end{proof}

An useful consequence of Lemma \ref{lem.C^1 map in the limit} is the following consistency  of the boundary maps: the boundary map $\xi_{\r_0}^2$  of a reductive limit of representations satisfying property $H_1$ has image in a single irreducible factor.

\begin{prop}\label{claim.weakly H_1 under weakly 2-Anosov}
If $\{\r_n\}$ satisfy property $H^*_1$,  $\r_0$ is reductive $2$--almost Anosov, and $\xi^{2}_{\r_n}$ converge pointwise to $\xi^{2}_{\r_0}$, then there is a $\rho_0$-irreducible subspace $E_1<E$ such that $ \xi^2_{\r_0}(x)<E_1$ for all $x\in\dg$.

\end{prop}
\begin{proof}
Assume by contradiction that the claim does not hold. As $\r_0$ is reductive and 2--almost Anosov, Lemma \ref{lem.splitting boundary map for reducible reps} implies that we can find a splitting
$$\r_0=\eta_1\oplus \eta_2:\G\to \PGL(E_1\oplus E_2)$$
such that, for all $s\in\dg$,  $ \xi^2_{\r_0}(s)$ meets non-trivially both factors. We first show that there is a continuous map $\dg\to\P(E_1)$, which is a pointwise limit of $\xi^1_{\r_n}$ on a dense set. We consider two cases:

\textbf{Case (1):} If there is a 1--proximal element $\g\in\G$ with attracting line in the irreducible space $E_1$, since the representation is 1-positively ratioed by Proposition \ref{prop.hyperconvex implies positively ratioed}, then Theorem \ref{thm.limits of 1-pos ratio}, applicable thanks to Corollary \ref{cor.weakly Anosov limit generalizing proximality}, 
 gives an equivariant continuous boundary map $\xi^1_{\r_0}:\dg\to\P(E_1)$ which is a limit of $\xi^1_{\r_n}$ on the dense set $\G \g_+\subset \dg$.\smallskip

\textbf{Case (2):}  For every $\g\in\G$, the generalized eigenspace $E_{\rho_0(\g)}^{\lambda_1}$ relative to eigenvalues of maximum modulus has dimension 2,  thus it is equal to $\xi^2_{\rho_0}(\g_+)$  and meets both factors $E_i$. Furthermore, since $\rho_n$ satisfies property $H^*_1$, the eigenvalues $\lambda_1(\r_n(\g))$ and $\lambda_2({\r_n}(\g))$ have the same sign \cite[Proposition 1.5]{Beyrer-Pozzetti}. Since $\rho_0(\gamma)$ fixes the intersections with the two factors  in $\P(\xi^2_{\rho_0}(\g_+))$,  {it  acts on $\xi^2_{\rho_0}(\g_+)$ as $\lambda_1(\rho_0(\g))\Id$}.

We choose a non-trivial element $\g\in \G$. Since $\xi_{\r_n}^{d-2}(\g_-)$ converges to $ \xi_{\r_0}^{d-2}(\g_-)$, we can assume, up to conjugating $\r_n$ by a sequence of elements $g_n\in\PGL(E)$ converging to the identity, that the subspace $\xi_{\r_n}^{d-2}(\g_-)$ is constant along the sequence and equals $\g_-^{d-2}:=\xi_{\r_0}^{d-2}(\g_-)$. We set $V_{\g}:=\quotient{E}{\g_-^{d-2}}$ and denote by
$$\mathtt{p}_{\g}:\P(E)\setminus\P(\g_{-}^{d-2})\to \P(V_{\g})$$ 
the quotient projection. Property $H_1$ implies that $\mathtt{p}_{\g}\circ \xi_{\r_n}^1|_{\dg\backslash \{\g_-\}}$ is a continuous and injective map from an interval to a circle.

We choose $x,z\in\dg\backslash \{\g_+,\g_-\}$ in the two connected components and pass to a subsequence so that $x^1_{\r_n},z^1_{\r_n}$ converge to points  $x^1_{\r_0}, z^1_{\r_0}$ respectively. Since $x^1_{\r_0}< x^2_{\r_0}$, we have $x^1_{\r_0} \notin \P\g_-^{d-2}$.  Set 
$$q_x:=\mathtt{p}_{\g}( x^1_{\r_0}) \in \P(V_{\g}).$$
In the same way define $q_z\in \P(V_{\g})$.
We claim that for any $y\in (\g_+,\g_-)_z$ and any subsequence such that $y^1_{\r_n}$ converges to a point $y^1_{\r_0}$, we have $\mathtt{p}_{\g}(y^1_{\r_0})=q_x$:  this is the case for any point  of the form $\g^k x$ since the action of $\g_{\r_0}$ on $\P(V_{\g})$ is trivial; for another point $y$, we find $k\in \Z$ such that $(x,y,\g^k x, \g^k y)$ is cyclically ordered. Since the maps $\mathtt{p}_{\g}\circ \xi_{\r_n}^1|_{\dg\backslash \{\g_-\}}$ are monotone,  $\mathtt{p}_{\g}(y^1_{\r_n})\to q_x$ for any $y\in (\g_+,\g_-)_z$ without passing to an additional subsequence. The analogue holds for $q_z$ and $(\g_+,\g_-)_x$.

We define a map $\xi^1_{\r_0}:\dg\backslash \{\g_+,\g_-\}\to \P(E)$ choosing the only line in $y^2_{\r_0}$ projecting to $q_x$ if $y\in (\g_-,\g_+)_z$ and projecting to $q_z$ if $y\in (\g_-,\g_+)_x$. This exists and is unique since $y^2_{\rho_0}$ is transverse to $\g_-^{d-2}$ and thus $\mathtt{p}_{\g}$ induces an isomorphism between $y^2_{\r_0}$ and $V_\g$.
%$$\xi^1_{\r_0}(y)<y^2_{\r_0},\quad \mathtt{p}_{\g}(\xi^1_{\r_0}(y))=q.$$
%The transversality of $y^2_{\r_0}$ and $\g_-^{d-2}$ implies that such an element exists and is unique. 
The map  $\xi^1_{\r_0}$ is continuous since $\xi^2_{\r_0}$ is, it is the pointwise limit of $\xi^1_{\r_n}$ since 
$\xi^2_{\r_n}\to \xi^2_{\r_0}$  and $\mathtt{p}_{\g}(\xi^1_{\r_n}(y))\to \mathtt{p}_{\g}(\xi^1_{\r_0}(y))$. %\Longrightarrow \xi^1_{\r_n}(y)\to \xi^1_{\r_0}(y).$$ 
Applying the same argument to a conjugate of $\g$  implies that $\xi^1_{\r_0}$ can be extended to a continuous map on neighbourhoods of $\g_-$ and $\g_+$ and thus defines a continuous map on all of $\dg$, which is the limit of $\xi^1_{\r_n}$.\smallskip

\textbf{Contradiction:} In both cases  Lemma \ref{lem.C^1 map in the limit} implies that $\xi^1_{\r_0}$ has $C^1$--image with tangent $\xi^2_{\r_0}$, furthermore  there is $\g\in\G$ such that $\mathtt{p}_{\g}\circ \xi^1_{\r_0}$ is constant. This is a contradiction:   the tangent spaces at  $\mathtt{p}_{\g}\circ \xi^1_{\r_0}(x)$ of  $\mathtt{p}_{\g}\circ \xi^1_{\r_0}(\dg)$ are non-degenerate because $\xi^2_{\r_0}(x)$ is transverse to $\xi^{d-2}_{\r_0}(\g_-)$. 
\end{proof}

We  now have all the ingredients to conclude the study of limits of representations satisfying property $H_1$.
\begin{prop}\label{prop.weakly H_1 under weakly 2-Anosov}
If $\{\r_n\}$ satisfy property $H^*_1$,  $\r_0$ is reductive $2$--almost Anosov, and $\xi^2_{\r_n}\to\xi^2_{\r_0}$, then $\r_0$ is $1$--Anosov. If in addition $\xi^{d-2}_{\r_n}\to\xi^{d-2}_{\r_0}$,  then the boundary maps of $\r_0$ satisfy property $H^*_1$.
\end{prop}
\begin{proof}
By Proposition \ref{claim.weakly H_1 under weakly 2-Anosov} there exists an irreducible subspace $E_1$ such that
$$\r_0=\eta_1\oplus \eta_2:\G\to \PGL(E_1\oplus E_2)$$ 
and $\xi_{\rho_0}^2(x)<E_1$ for all $x\in\dg$. Let us choose points $\bx,\by$. Up to passing to a subsequence, the limits of $\xi^1_{\rho_n}(\bx)$ and of $\xi^{d-1}_{\rho_n}(\by)$ exist and  satisfy $\xi^1_{\r_n}(\bx)\in \P(E_1)$ and  $E_2<\lim \xi^{d-1}_{\r_n}(x)$; Proposition \ref{pro.condition L irreducible} implies that condition (Tr) holds with respect to  $\bx,\by\in\dg$ and hence there exist  continuous, equivariant, transverse boundary maps for $\rho_0$ that are contained in $\P(E_1)$ (Theorem \ref{thm.limits of 1-pos ratio}). These boundary maps restrict to boundary maps of $\eta_1$ and thus the irreducibility implies that $\eta_1$ is $1$--Anosov. Since $\xi^2_{\r_0}$ is dynamics preserving and contained in $E_1$, the first gap of $\rho_0$ is the first gap of $\eta_1$, and therefore $\rho_0$ is 1--Anosov as well.

The boundary maps of $\rho_0$ satisfy property $H_1$: otherwise, since  they are limits of boundary maps satisfying $H_1$,  there would be  $x\in\dg$ and an open set $U\subset \dg$ such that $u^1_{\r_0}+x_{\r_0}^{d-2}=v^1_{\r_0}+x_{\r_0}^{d-2}$ for all $v,u\in U$ (cfr. \cite[Proposition 9.3]{PSW1} and \cite[Proposition 4.7]{Beyrer-Pozzetti}). As $E_2\subset x_{\r_0}^{d-2}$ and  $\xi^1_{\r_0}(\dg)=\xi^1_{\eta_1}(\dg)\subset \P(E_1)$,  also 
$u^1_{\eta_1}<v^1_{\eta_1}+x_{\eta_1}^{d_1-2}$,  for some $v\in U$ fixed and all $u\in U$. This contradicts irreducibility of $\eta_1$.
Property $H_{d-1}$ follows in the same way using the dual representation $\r^{\dual}$.
\end{proof}

The general case follows using exterior powers:
\begin{cor}\label{cor.closdeness Hk under Anosov}
If $\{\r_n\}$ satisfy property $H^*_k$, $\r_0$ is reductive $\{k-1,k+1\}$--almost Anosov, and $\xi^{j}_{\r_n}\to \xi^{j}_{\r_0}$ for $j=k\pm 1,d-k\pm 1$. Then $\r_0$ satisfies property $H^*_k$.
\end{cor}

\begin{proof}
By Proposition \ref{prop.Anosov for exterior power} the representations $\wedge^k \r_n$ are $2$--Anosov and the representation $\wedge^k \r_0$ is almost $2$--Anosov. Since $\xi^{j}_{\r_n}\to \xi^{j}_{\r_0}$ for $j=k\pm 1,d-k\pm 1$,
$$\xi^2_{\wedge^k\r_n}(x)\to \xi^2_{\wedge^k\r_0}(x),\quad \xi^{d_k-2}_{\wedge^k\r_n}(x)\to \xi^{d_k-2}_{\wedge^k\r_0}(x)$$
for all $x\in\dg$ (here $d_k=\dim \wedge^k E$). Thanks to Proposition \ref{prop.weakly H_1 under weakly 2-Anosov},  $\wedge^k \r_0$ is $1$--Anosov and the boundary maps satisfy property $H_1^*$. By Propositions \ref{prop.Anosov for exterior power} and \ref{prop.Hk implies H1 for exterior power}, $\r_0$ is $k$--Anosov and the boundary maps satisfy $H_k^*$.
\end{proof}

Since property  $H^*_k$ is a conjugation invariant, open condition, $\r$ satisfies property $H^*_k$ if a semisemplification $\r^{ss}$ does. 
We can now show that the converse is also true:
\begin{cor}\label{cor.Hk and semisimplification}
A representation $\r:\G\to \PGL(E)$ satisfies property $H^*_k$ if and only if a semisimplification $\r^{ss}:\G\to\PGL(E)$ satisfies property $H^*_k$.
\end{cor}

\begin{proof}
Assume that $\r$ satisfies property $H^*_k$. Choose $g_n\in\PGL(E)$ so that $\{g_n\r g_n^{-1}\}_{n\in\N}$  converges to $\r^{ss}$; since $g_n\r g_n^{-1}$ have property $H^*_k$, and the representation $\r^{ss}$ is reductive and $\{k-1,k,k+1\}$--Anosov, Corollary \ref{cor.closdeness Hk under Anosov} implies that $\r^{ss}$ satisfies property $H^*_k$.
\end{proof}

\section{$k$-positive representations}\label{s.k-pos}
In this section we study $k$--positive representations, a rich class of representations that satisfy positivity properties for some projections, and that includes small deformations of various Fuchsian loci.

\subsection{Positivity on the space of full flags}\label{s.5.1}
The space of full flags carries a positive structure, which we briefly recall here. For more details see e.g. \cite[Section 5]{FG}.

Let $E$ be a real $d$--dimensional vector space with $d\geq 2$. We fix a basis $(e_1,\ldots,e_d)$ and denote by $\sX,\sZ\in \calF(E)$ the full flags defined by $$\sX^j=\<e_d,\ldots,e_{d-j+1}\>,\quad \sZ^j=\<e_1,\ldots,e_j\>.$$
Let $U$ be the unipotent radical of the stabilizer of $\sZ$ in $\PGL(E)$, namely the group of upper triangular matrices with ones on the diagonal. Observe that a flag $Y\in \calF(E)$ transverse to $\sZ$ can be written as $u\cdot \sX$ for a unique $u\in U$. The \emph{positive semigroup} $U^{>0}<U$ is  the set of matrices  all whose minors that do not necessarily vanish %by the condition of being unipotent
 are positive - one can show that $U^{>0}$ is indeed a (sub-)semigroup.
\begin{defn}
A triple $(x,y,z)\in \calF^3(E)$ is \emph{positive} if there exists $g\in\PGL(E)$ and a positive element $P\in U^{>0}$ such that 
$$(gx,gy,gz)=(\sX, P\sX,\sZ).$$
A $n$-tuple $(x_1,\ldots, x_n)\in \calF^n(E)$ is \emph{positive} if there exist $P_2,\ldots,P_{n-1}\in U^{>0}$, and $g\in\PGL(E)$ such that
\begin{align}\label{e.posntuple}
(gx_1,gx_2,\ldots, gx_{n-1},gx_n)=(\sX, P_2\sX,\ldots, P_2\cdots P_{n-1}\sX,\sZ).
\end{align}
\end{defn}

Given a flag $F\in\calF(E)$ we denote by $\Omega_F$ the set of flags transverse to $F$. We will need the following fact from \cite{Lusztig}, which is a crucial property of positive triples.

\begin{thm}\label{t.ccomp}
The set 
$$\{F\in\calF(E)|F =P\sX, P\in U^{>0} \}$$
is a connected component of $\Omega_{\sX}\cap\Omega_{\sZ}$.
In particular, given two transverse flags $A,B$, the set 
$$\{F\in\calF(E)|(A,F,B) \text{ is positive }\}$$
is a union of connected components of $\Omega_A\cap\Omega_B$.
\end{thm}
We record the following useful corollary.
\begin{cor}\label{cor.deform positive tuples}
Let $c:[a,b]\subset \R\to \calF^l(E), c(t)=(x_1(t),\ldots,x_l(t))$ be a continuous path such that $x_{i}(t)\tv x_j(t)$ for $i\neq j$ and all $t\in [a,b]$. If $c(a)$ is a positive $l$--tuple, then $c(t)$ is a positive $l$--tuple for all $t\in [a,b]$.
\end{cor}

{{As an immediate consequence of Corollary \ref{cor.deform positive tuples} and the connectedness of the set of positively oriented $n$--tuples in $(\mathbb{S}^1)^{(n)}$ we get
 \begin{cor}\label{cor.pos tuples}
Let $\xi:\mathbb{S}^1\to \calF(E)$ be a continuous transverse curve. If the image of one positively oriented $n$-tuple in $ \mathbb{S}^1$ under $\xi$ is a positive $n$-tuple in $\calF^n(E)$, then every positively oriented $n$-tuple in $ \mathbb{S}^1$ is mapped to a positive $n$-tuple in $\calF^n(E)$.
\end{cor}
 
Several additional transversality properties follow form the definition of the positive semigroup.

\begin{prop}\label{prop.transversalities of pos triples}
Let $(x,y,z)\in \calF^3(E)$ be a positive triple. Then 
\begin{enumerate}
\item $x^j+y^k+z^l$ is direct if $j+k+l= d$
\item $(z^{d-j+1}\cap x^{j})+(z^{d-j+1}\cap y^{j})+z^{d-j-1}$ is direct for all $1\leq j\leq d-1$
\item $(z^{d-j+1}\cap x^{j+1})+(z^{d-j+1}\cap y^{j})+z^{d-j-2}$ is direct for all $1\leq j\leq d-2$.\footnote{Here, as always, we set $z^d=E$ and $z^{0}=\{0\}$}
\end{enumerate}
\end{prop}

\begin{proof}
Since the action of $\PGL(E)$ preserves transversality, it is enough to consider triples of the form $(\sX,P\sX,\sZ)$ for $P\in U^{>0}$. For those triples the transversalities are an immediate consequence of the positivity of certain minors of $P$. More specifically:\\
$(1)$ holds because the minor of $P$ obtained from deleting the first $l+j$ columns, the first $l$ rows and the last $j$ rows is positive.\\
$(2)$ holds because the matrix coefficient $P_{d-j,d-j+1}$ is positive.\\
$(3)$ holds because the matrix coefficient $P_{d-j-1,d-j+1}$ is positive.
\end{proof}

The above proposition is independent of the order of $x,y$ and $z$ because any permutation of a positive triple is again a positive triple:

\begin{prop}[{\cite[Theorem 5.2]{FG}}]
Let $(x_1,x_2,x_3)\in \calF^3(E)$ be a positive triple. Then $(x_{\s(1)},x_{\s(2)},x_{\s(3)})$ is positive triple for any permutation $\sigma$.
\end{prop}

We come to the definition of a positive curve.

\begin{defn}
A map $\xi:I\subset \mathbb S^1\to \calF(E)$ is \emph{positive} if, for every positively oriented $n$-tuple $(x_1,\ldots,x_n)\in I$, the $n$-tuple $(\xi(x_1),\ldots,\xi(x_n))\in\calF^n(E)$ is positive. 
 \end{defn} 
 
The following crucial fact about projections follows from \cite[Lemma 9.8]{FG}. Given a map $\xi:I \to \calF(E)$ we denote by $\xi^k:I\to\Gr_k(E)$ the induced map.

\begin{prop}\label{prop.positive-quotient}
For $I\subset\mathbb{S}^1$ let $\xi:I \to \calF(E)$ be a positive curve. Then for any $x\in I$, and any index $j$, the $\xi^j(x)$-projection 

$$\begin{array}{rl}
\pi_{[x^j]}:&I\to \calF\left(\quotient{E}{\xi^j(x)}\right)\\
	&\left\{\begin{array}{l}
	\pi_{[x^j]}(\xi(y)):= ([\xi^1(y)],\ldots,[\xi^{d-j-1}(y)]),\quad y\neq x\\
	\pi_{[x^j]}(\xi(x)):= ([\xi^{j+1}(x)],\ldots,[\xi^{d-1}(x)])
	\end{array}\right.
\end{array}$$
is a positive curve.
\end{prop}

\begin{proof}
The statement for $\pi_{[x^1]}$ is a direct application of \cite[Lemma 9.8]{FG}.
Observe that $\quotient{(y^k+x^1)}{x^1} + \quotient{x^{j+1}}{x^1}=\quotient{(y^k+x^{j+1})}{x^1}$ for $y\neq x\in I$. Therefore $\pi_{[x^j]}\circ \pi_{[x^1]}=\pi_{[x^{j+1}]}$, and  $\pi_{[x^j]}=\pi_{[x^1]}\circ \ldots\circ \pi_{[x^1]}$. The claim thus follows in full generality. 
\end{proof}

%%%%%%%%%%
\subsection{$k$-positive representations}
%%%%%%%%%%%
The following curves, naturally associated to an Anosov representation, will be used in the definition of $k$-positive representations:
\begin{defn} Let $\rho$ be $\{1,\ldots, k\}$--Anosov and $x\in \dg$. 
\begin{itemize}
\item The \emph{$x^{d-k}$-projection} is the curve
\begin{align*}
\pi_{[x]}:&\dg\backslash\{x\}\to \calF(X)\\
\pi_{[x]}(y):&= ([y^1]_X,\ldots,[y^{k-1}]_X),
\end{align*}
where $[\cdot]_X$ is the natural projection to $X:=\quotient{E}{x^{d-k}}$.

\item The \emph{$x^k$-truncation} is the curve 
\begin{align*}
\pi_{x}:&\dg\backslash\{x\}\to \calF(x^k)\\
\pi_{x}(y):&= (y^{d-k+1}\cap x^k,\ldots,y^{d-1}\cap x^k).
\end{align*}
\end{itemize}
{{When not otherwise clear, we will also write $\pi_{[x^{d-k}]}$ and $\pi_{x^k}$.}}
\end{defn}

\begin{defn}
A representation $\r:\G\to \PGL(E)$ is  \emph{$k$--positive} for $k\geq 2$ if it is $\{1,\ldots, k\}$--Anosov and for any $x\in \dg$ the $x^{d-k}$-projection $\pi_{[x]}$ and the $x^k$-truncation  $\pi_x$
are positive.
Moreover $\r$ is a \emph{1--positive representation} if it is 1--Anosov.\footnote{There is no positive structure on $\quotient{E}{x^{d-1}}$.}
\end{defn}
Observe that being $k$--positive is conjugation invariant.
Truncations and projections are linked through duality.
The \emph{dual (or contragradient) representation} $\tilde{\r}^{\dual}:\G\to \GL(E^{*})$
of a representation $\tilde{\rho}:\G\to \GL(E)$ is defined by the relation $(\tilde{\r}^{\dual}(\g)(w^{*}) )(v)= w^{*} (\tilde{\r}(\g)^{-1}v)$ for all $\g\in\G,w^{*}\in E^{*}$ and $v\in E$. This also defines the dual $\r^{\dual}:\G\to \PGL(E^{*})$
of a representation $\r:\G\to \PGL(E)$. 
The representation $\rho$ is Anosov if and only if its dual $\rho^\dual$ is, furthermore the boundary maps for $\rho^\dual$ can be explicitly constructed from the maps for $\rho$, as we now recall. %satisfy $x^k_{\rho^\dual}=(x^{d-k}_\rho)^\dual$.
Given  $W\in \Gr_k(E)$ we  denote by $W^{\dual}\in \Gr_{d-k}(E^{*})$ its annihilator:  
$$W^{\dual}:=\{w^{*}\in E^{*}|\; w^{*}(W)=0\}.$$ 
If $\xi^k:\dg\to \Gr_k(E)$ and $\xi^{d-k}:\dg\to \Gr_{d-k}(E)$ are $\rho-$equivariant, continuous and transverse, then the maps $\xi_{\dual}^{k}:\dg\to \Gr_{k}(E^{*})$ and $ \xi_{\dual}^{d-k}:\dg\to \Gr_{d-k}(E^{*})$  defined by $\xi^s_\dual(x)=(\xi^{d-k}(s))^\dual$ are $\rho^{\dual}-$equivariant continuous and transverse. Furthermore $\xi^{k},\xi^{d-k}$ are dynamics preserving if and only if $\xi_{\dual}^k,\xi_{\dual}^{d-k}$ are. 
The space  $\xi^k_\dual(x)=(\xi^{d-k}(x))^\dual$ is canonically identified with $\left(\quotient{E}{x^{d-k}}\right)^*$. Under this identification the curve $\pi^{\r}_{[x]}$ is dual to the curve $\pi^{\r^{\dual}}_x$.
As a result we obtain:
\begin{lem}\label{lem.positivity and dual}
Let $\r$ be $\{1,\ldots,k\}$--Anosov  and $\r^{\dual}$ be its dual represntation. Then the curve $\pi^{\r}_{[x]}$ is positive if and only if the curve $\pi^{\r^{\dual}}_x$ is positive.
\end{lem}

Examples of partially positive representations can be constructed through Fuchsian embeddings. Since $k$--positivity is preserved under small deformations (Corollary \ref{cor.k-pos open} below), we obtain a large class of examples.
\begin{example}\label{e.Fuchpos}
In the notation of Example \ref{ex.Fuchsian I}, the composition $\tau_{\un d}\circ \r_{hyp}$ is $k$--positive for $k$ with $k<\frac{1}{2}(d_1-d_2)+1$.
Indeed such a representation is $\{1,\ldots,k\}$--Anosov (Example \ref{ex.Fuchsian I}), its Anosov boundary maps take value in the partial flags of $\R^{d_1}\subset \R^d$ and are the partial flags, for the indices $\{1,\ldots,k,d_1-k,\ldots,d_1\}$, of the  Veronese curve in $\R^{d_1}$. Since the Veronese curve is positive, Proposition \ref{prop.positive-quotient} implies that the $x^{d-k}$-projection $\pi_{[x]}$ is positive. Since  $\tau_{\un d}\circ \r_{hyp}$ is self-dual, also the $x^{k}$-truncation $\pi_x$ is positive by Lemma \ref{lem.positivity and dual}.
\end{example}

\subsection{Transversality properties of $k$--positive representations}
We give here several transversality properties of $k$--positive representations. To derive those properties we first need to show that the positive curves $\pi_x$ and $\pi_{[x]}$ extend to continuous curves defined on all of $\dg$.

Recall the following notion introduced by Labourie  \cite[Section 8]{Labourie-IM}:
\begin{defn}
\leavevmode
\begin{itemize}
 \item Continuous, equivariant boundary maps $\xi^{j}:\dg\to \Gr_j(E)$ for $j=p,q,l$ are  \emph{$(p,q,l)$--direct} if the  sum 
$\xi^p(x)+\xi^q(y)+\xi^l(z)$
is direct for all $(x,y,z)\in\dg^{(3)}$. 
\item A representation $\r:\G\to\PGL(E)$ is \emph{$(p,q,l)$--direct} if it is $\{p,q,l\}$--Anosov and the  boundary maps are $(p,q,l)$--direct. 
\end{itemize}
\end{defn}

The boundary maps of $(p,q,d-p-q)$--direct representations satisfy an important continuity property.

\begin{prop}[{\cite[Theorem 7.1]{PSW1}, cfr. also \cite[Proposition 21]{Guichard}}]\label{prop.continuity of sums}
Let $\r$ be $(p,q,d-p-q)$--direct. Then%\footnote{Observe that order of $p,q,l$ in $(p,q,l)$--direct does not matter, i.e. we get continuity also for other exponents.}
$$y^p+w^{q}\to x^{p+q}$$
if $y,w\to x$.
\end{prop}

\begin{lem}\label{5.9}
Let $\r$ be $k$--positive. Then $\r$ is $(j,k-j,d-k)$--direct for all $j\leq k$. In particular $\pi_{[x]}$ extends to positive continuous curve defined on all of $\dg$ by
$$\pi_{[x]}(x):= ([x^{d-k+1}]_X,\ldots,[x^{d-1}]_X)$$
\end{lem}

\begin{proof}
Since a positive curve is transverse, the sum $\pi_{[x]}(y)^j+ \pi_{[x]}(z)^{k-j}$ is direct for all $(x,y,z)\in \dg^{(3)}$. Equivalently the sum $y^j+z^{k-j}+x^{d-k}$ is direct for all $j\leq k$, namely $\r$ is $(j,k-j,d-k)$--direct for all $j\leq k$.  

Proposition \ref{prop.continuity of sums} then gives $y^j+x^{d-k}\to x^{d-k+j}$ if $y\to x$ in $\dg$, which implies that $\pi_{[x]}$ is continuous at $x$. The transversality of the Anosov boundary maps implies that $\pi_{[x]}(x)\tv \pi_{[x]}(y)$ for all $y\neq x$. By Corollary \ref{cor.pos tuples}, the projection $\pi_{[x]}$ extends to a positive map on all  $\dg$.
\end{proof}

\begin{cor}\label{cor.k-pos implies j-pos}
Let $\r$ be $k$--positive. Then $\r$ is $j$--positive for all $j\leq k$.
\end{cor}
\begin{proof}
Since the projection $\pi_{[x^{d-k}]}:\dg\to \calF(\quotient{E}{x^{d-k}})$ defines a positive curve, we can apply Proposition \ref{prop.positive-quotient} to derive that the projection $\pi_{[x^j]}$ applied to the curve $\pi_{[x^{d-k}]}$ gives rise to a positive curve in $\calF(\quotient{E}{x^{d-k+j}})$. Since $\pi_{[x^j]}\circ \pi_{[x^{d-k}]}=\pi_{[x^{d-k+j}]}$ for all $j<k-1$, the projection $\pi_{[x^{d-k+j}]}$ defines a positive curve. The analogous statement for the truncation $\pi_x$ follows using Lemma \ref{lem.positivity and dual}.
\end{proof}

As a consequence we obtain the following  hyperconvexity property of $k$-positive representations, which was proven by Labourie for Hitchin representations:
\begin{prop}\label{prop.hyperconvex properties}
Let $\r$ be  $k$--partially positive. Then the sum
$$x_1^{n_1}+ \ldots +x_i^{n_i}$$
is direct if $n_1+\ldots+n_{i-1}\leq d$ and $n_i\geq d-k$.
\end{prop}
\begin{proof}
By Corollary \ref{cor.k-pos implies j-pos}, the $k$--partially positive representation $\rho$ is also $n_i$--partially positive, and thus the $x_i^{n_i}$-projection $\pi_{[x_i]}: \dg\to  \calF(\quotient{E}{x_i^{n_i}})$ defines a positive curve. We can then iteratively apply Proposition \ref{prop.positive-quotient} to get that the projection  
$$\pi=\pi_{[x_4^{n_4}]}\circ \ldots\circ \pi_{[x_i^{n_i}]}: \dg\to  \calF\left(\quotient{E}{ x_4^{n_4}+ \ldots +x_i^{n_i}}\right)$$ 
is also positive. As a positive triple of full flags is $(p,q,l)$--direct for all values (Proposition \ref{prop.transversalities of pos triples}$(1)$),  the sum $\pi(x_1^{n_1})+\pi(x_2^{n_2})+\pi(x_3^{n_3})$ is direct. This  implies that the sum
$$x_1^{n_1}+ \ldots +x_i^{n_i}$$
is direct.
\end{proof}

\begin{prop}\label{prop.k-pos and Hk+Ck}
A $k$--positive representation $\r$ satisfies property $H^*_j$ for $j=1,\ldots,k-1$ and property $C_j^*$ for $j=1,\ldots,k-2$.
\end{prop}

\begin{proof}
Let $(x,y,z)\in \dg^{(3)}$. The sum $(y^{j}\cap x^{d-j+1})+(z^{j}\cap x^{d-j+1})+x^{d-j-1}$ is direct if and only if the sum
$$(\pi_{[x]}(y)^j\cap \pi_{[x]}(x)^2)+ (\pi_{[x]}(z)^j\cap \pi_{[x]}(x)^2)$$
is direct, for the $d-j-1$-projection $\pi_{[x]}:\dg\to \calF(\quotient{E}{x^{d-j-1}})$. For $j\leq k-1$ the latter sum is direct by Corollary \ref{cor.k-pos implies j-pos} and $(2)$ of Proposition \ref{prop.transversalities of pos triples}. In particular $\r$ satisfies property $H_j$ for $j=1,\ldots,k-1$. Lemma \ref{lem.positivity and dual} gives that it also satisfies property $H^*_j$. %Proposition \ref{prop.hyperconvex implies positively ratioed} yields that $\r$ is positively ratioed.

The representation $\r$ satisfies property $C_j^*$ for $j=1,\ldots,k-2$ thanks to the same argument and Proposition \ref{prop.transversalities of pos triples} $(3)$. 
\end{proof}
Thus Theorems \ref{prop.hyperconvex implies positively ratioed} and \ref{thm.collar lemma} apply and we have
\begin{cor}\label{c.kpospos}
	A $k$--positive representation $\r$ is $j$--positively ratioed for $j=1,\ldots,k-1$, and satisfies the $j$--collar lemma for $j=1,\ldots,k-2$.
\end{cor}	

Properties $H_j^*$ guarantee that positivity happens in a single irreducible factor of a reductive representation:
\begin{defn}\label{defn.coherence}
A $\{1,\ldots,k\}$--almost Anosov representation $\r:\G\to \PGL(E)$ is \emph{$k$--coherent} if it splits as $\r=\eta_1\oplus \eta_2\to \PGL(E_1\oplus E_2)$, where $\eta_1$ is irreducible, and $\xi^j(\dg)\subset \Gr_j(E_1)$ for all $j=1,\ldots,k$.
\end{defn}
By Lemma \ref{lem.splitting boundary map for reducible reps} $k$--coherence is equivalent to the requirement that the factor $E_2$ is contained in $\xi^{d-j}(x)$ for all $x\in\bord\G$ and all  $j=1,\ldots,k$.
Coherence holds for $k$--positive representations:
\begin{cor}\label{cor.coherence}
A $k$--positive reductive representation $\r$ is $k$--coherent.
\end{cor}

\begin{proof}
As $\r$ is reductive, it splits as $\r=\eta_1\oplus \eta_2\to \PGL(E_1\oplus E_2)$, where $\eta_1$ is irreducible, and $\xi^1(\dg)\subset \P(E_1)$. Recall from Lemma \ref{lem.splitting boundary map for reducible reps} that, since $\r$ is $s$-Anosov, for $s\leq k$ the dimension $s_1=\dim(\xi^s(x)\cap E_1)$ is constant, furthermore we deduce from Proposition \ref{prop.Hk and reducible} and by induction on $s$ that, since the representation satisfies property $H_s$, then $(s+1)_1=s_1+1$.
\end{proof}

\subsection{$k$-tridirect representations}\label{s.tridirect}
We now introduce a second class of representations, generalizing the notion of 3-hyperconvexity introduced by Labourie  \cite[Section 8]{Labourie-IM}:
\begin{defn}
\leavevmode
\begin{itemize}
\item A representation is \emph{$k$--tridirect} if it is $\{1,\ldots,k\}$--Anosov and  $(p,q,l)$--direct for all $(p,q,l)$ with $l\geq d-k$ (and $p+q+l\leq d$). 
\item A representation $\r$ is \emph{$k^*$--tridirect} if $\r^{\dual}$ is $k$--tridirect. 
\end{itemize}
\end{defn}

From Corollary \ref{cor.k-pos implies j-pos} and Lemma \ref{5.9} we immediately get.

\begin{prop}\label{prop.k-pos implies tridirect}
A $k$--positive representation is $\{k,k^*\}$--tridirect.
\end{prop}

The following property of tridirect representations is a well known consequence of the fact that $\G$ acts properly and cocompactly on $\dg^{(3)}$, transversality is an open condition and the Anosov boundary map varies continuously with the representation (cfr. \cite[Proposition 6.2]{PSW1}):

\begin{prop}\label{lem.direct reps are open}
Being $(p,q,l)$--direct is an open property in $\Hom(\G,\PGL(E))$ and thus also being $\{k,k^*\}$--tridirect.
\end{prop}

Furthermore we have

\begin{prop}\label{prop.hypconv open among direct reps}
$k$--positive representations form connected components of $\{k,k^*\}$--tridirect representations.
\end{prop}

\begin{proof}
Let $t\to \r_t$ for $t\in [0,1]$ be a continuous path in $\Hom(\G,\PGL(E))$ consisting of $k$--tridirect representations. Choose a point $x\in\dg$. By conjugating $\r_t$ with $g_t\in \PGL(E)$, where $g_t$ depends continuously on $t$ and $g_0=\rm{id}$, we can assume that $x_{\r_t}^{d-k}=x_{\r_0}^{d-k}$ for all $t\in [0,1]$. The projections $\pi_{[x]}^t$ associated to the representations $\r_t$ vary continuously in $t$ and $\dg\backslash\{x\}$. Moreover $\pi_{[x]}^t(y)\tv \pi_{[x]}^t(z)$ for all $z\neq y$, since the representations are $k$--tridirect. By Corollary \ref{cor.deform positive tuples} this is enough to guarantee that for every $t\in [0,1]$ the curve $\pi_{[x]}^t$ is positive, provided $\pi_{[x]}^0$ is positive.  Using $\r_t^{\dual}$ we obtain that also $\pi^t_x$ is positive.
\end{proof}

\begin{cor}\label{cor.k-pos open}
Being $k$--positive is an open property in $\Hom(\G,\PGL(E))$.
\end{cor}

%%%%%%%%%%%%%%%%%%%%%%%%%%%%%%%%%%%%%%%%%%%%%%

\section{Degenerations of $k$--positive representations}\label{s.6}
In this section we study degenerations of $k$--positive representations and prove Theorem \ref{thm.intro-degenerations}. In Section \ref{s.degen-Anosov property} we show that $k$--positive representations degenerate if and only if they lose Anosovness. 
In Section \ref{s.irred limits} we prove Theorem \ref{thm.intro-degenerations} under the additional assumption that the boundary maps in the limit are coherent. The latter is shown in Section \ref{s.tecnical} to complete the proof.

\subsection{Anosov limits of $k$-positive representations}\label{s.degen-Anosov property}
\begin{prop}\label{prop.direct reps closed among Anosov}
Let $\{\r_n:\G\to\PGL(E)\}_{n\in \N}$ be $k$--positive and converge to a reductive $\{1,\ldots,k\}$--Anosov representation $\r_0$. Then $\r_0$ is $k$--tridirect.
\end{prop}

\begin{proof}
We assume by contradiction that $\r_0$ is not $(p,q,l)$- direct for some $(p,q,l)$ with $p+q\leq d-l$ and $l\geq d-k$. It is enough to consider triples of the form $(p,q,d-p-q)$.
Let $(p,q)$ be minimal in the sense that $\r_0$ is $(p',q',d-p'-q')$--direct for $p'\leq p$, $q'\leq q$ and at least one of the two inequalities is strict and $\r_0$ is not $(p,q,d-p-q)$--direct. As $\r_0$ is $\{1,\ldots,k\}$--Anosov, $p\geq 1, q\geq 1$.

Let $(x,y,z)\in\dg^{(3)}$ such that 
$$x_{\r_0}^p+y_{\r_0}^q+z^{d-p-q}_{\r_0}$$
is not direct. Up to conjugating $\r_n$ with $g_n\in \PGL(E)$ converging to the identity, we can assume that the codimension 2 subspace
$$V:=x_{\r_n}^{p-1}+y_{\r_n}^{q-1}+z^{d-p-q}_{\r_n}=x_{\r_0}^{p-1}+y_{\r_0}^{q-1}+z^{d-p-q}_{\r_0},$$
doesn't depend on $n$. The sum on the right hand side is direct by  minimality of $(p,q)$. By Propositions \ref{prop.hyperconvex properties} and \ref{prop.continuity of sums} the map $\xi_n^V:\dg\to \P(\quotient{E}{V})\cong\R\P^1$ defined by
$$\left\{\begin{array}{l}
\xi_n^V(u):= [u_{\r_n}^1],\quad u\neq x,y,z\\
\xi_n^V(x):= [x_{\r_n}^p],\quad\\
\xi_n^V(y):= [y_{\r_n}^q],\quad \\
\xi_n^V(z):= [z_{\r_n}^{d-p-q+1}]
\end{array}\right.$$
is injective and continuous because $u_{\r_n}^1+x_{\r_n}^{p-1}\to x_{\r_n}^{p}$ if $u\to x$, and the same for $y,z$.  By the minimality assumption and the transversality of $\xi_{\rho_0}$, $[x_{\r_0}^p]=[y_{\r_0}^q]\neq [z_{\r_0}^{d-p-q+1}]\in \P(\quotient{E}{V})$. 
Since %Then the continuity of the boundary maps associated to $\{\r_n\}$, i.e. 
$u_{\r_0}^1=\lim u_{\r_n}^1$, for $u\in (x,y)_z$ either $[u_{\r_0}^1]=[x_{\r_0}^p]$ or $u_{\r_0}^1\in V$. In particular, for all $u\in (x,y)_z$, the line $u_{\r_0}^1$ is contained in the hyperplane $x_{\r_0}^{p}+y_{\r_0}^{q-1}+z^{d-p-q}_{\r_0}$.

By Corollary \ref{cor.closdeness Hk under Anosov}, $\r_0$ satisfies property $H_j$ for $j=1,\ldots,k-1$. Thus, by Proposition \ref{prop.Hk and reducible}, $\r_0$ is  $k$-coherent (Definition \ref{defn.coherence}) with induced splitting $\r_0=\eta_1\oplus \eta_2$. 
By Lemma \ref{lem.GW strong irreducibility}  $\eta_1$ is strongly irreducible. However, since $E_2\subset z^{d-p-q}_{\r_0}$, for all $u\in (x,y)_z$, $u_{\eta_1}^1$ is contained in the hyperplane $x_{\eta_1}^{p}+y_{\eta_1}^{q-1}+z^{d_1-p-q}_{\eta_1}$, where $d_1=\dim E_1$. This contradicts strong irreducibility of $\eta_1$  (Proposition \ref{prop.labourie-irreducible}) and concludes the proof.
\end{proof}

We get the following important consequence.

\begin{cor}\label{cor.par-pos closed among Anosov}
Being $k$--positive is open and closed among $\{1,\ldots,k\}$--Anosov representations.
A representation $\r$ is $k$--positive if and only if a semi-simplification $\r^{ss}$ is.
\end{cor}

\begin{proof}
We know already that $k$--positivity is an open property (Corollary \ref{cor.k-pos open}). Proposition \ref{prop.hypconv open among direct reps} and Proposition \ref{prop.direct reps closed among Anosov} imply that a reductive $\{1,\ldots,k\}$--Anosov limit of $k$--positive representations is $k$--positive. As in the proof of Corollary \ref{cor.Hk and semisimplification}, we deduce that $\r$ is $k$--positive if and only if a semi-simplification $\r^{ss}$ is% $k$---positive
, which concludes the proof.
\end{proof}

%%%%%%%%%%
\subsection{Limits of $k$-positive representations}\label{s.irred limits}
%%%%%%%%%%%%

	It follows from Corollary \ref{cor.par-pos closed among Anosov} that it is enough to study reductive limits:
\begin{prop}
In order to prove Theorem \ref{thm.intro-degenerations} it is enough to show that, for a limit $\rho$ of $k$-positive representations
\begin{itemize}
\item If $\rho$ is reductive, then it is $\{1,\ldots,k-3\}$-Anosov,
\item If $\rho$ is irreducible, then it is $\{1,\ldots,k-1\}$-Anosov,
\item If $\rho$ is reductive, $s$-Anosov, then it is $\{1,\ldots,s\}$-Anosov, $s\leq k$.
\end{itemize}
\end{prop}
\begin{proof}
	Let $\rho_n$ be a sequence of $k$-positive representations, converging to a representation $\rho$.
	We can conjugate each $\r_n$ such that the new sequence $\{\r_n'\}$ converges to the semisimplification $\r^{ss}$, which is reductive. Observe that all $\r_n'$ are $k$--positive, as this is a conjugation invariant condition. The first part of Corollary \ref{cor.par-pos closed among Anosov}, gives that if $\r_0^{ss}$ is $\{1,\ldots, j\}$--Anosov, then it is $j$--positive, the second part ensures that then also  $\r_0$ is.
\end{proof}	

In order to prove Theorem \ref{thm.intro-degenerations}, we will prove the three claims respectively in Proposition \ref{prop.limit of k-par hyp} (for general limits), Proposition \ref{prop.limits - strongly irred} (for irreducible limits), and Propositions \ref{cor.k-1 Anosov limit} and \ref{prop.k Anosov limt} (for $s$-Anosov limits). 
		To simplify the notation of the rest of this subsection we set
\begin{notation}
	In the rest of the subsection $\{\r_n:\G\to\PGL(E)\}_{n\in\N}$ is a sequence of $k$--positive representations converging to a reductive representation $\r_0$.
\end{notation}

We already have all the ingredients needed to study irreducible limits. As already remarked in the introduction, in this case we do not need collar lemmas. 

\begin{prop}\label{prop.limits - strongly irred}
If $\rho_0$ is irreducible, then it is $\{1,\ldots,k-1\}$--Anosov.
\end{prop}

\begin{proof}
{{ Corollary \ref{c.kpospos} guarantees that the representations $\rho_n$ are $\{1,\ldots,k-1\}$--positively ratioed; since $\rho_0$ is irreducible, Proposition \ref{prop.irred limits of 1-pos ratioed} applies and gives that $\rho_0$ is 1--Anosov. Thus Lemma \ref{lem.GW strong irreducibility} implies that $\r_0$ is strongly irreducible.}} Corollary \ref{cor.strongirredk-1Anosov} guarantees, by induction, that $\r_0$ is $\{1,\ldots,k-1\}$--Anosov.
\end{proof}

	In the study of general limits, we use collar lemmas (Corollary \ref{c.kpospos}) to ensure the existence of boundary maps: the following is a  consequence of 	Corollary \ref{cor.collar lemma almost Anosov}.
	\begin{prop}\label{lem.almost Anosov limit}
		The limit representation $\r_0$ is $\{1,\ldots,k-2\}$--almost Anosov.
	\end{prop}

The main technical step to improve from almost Anosov to Anosov is to show coherence of the limit representation, namely that all boundary maps are contained in a single irreducible factor (Definition \ref{defn.coherence}).
		
	\begin{prop}\label{lem.limit boundary map in the same factor}
			The limit representation $\r_0$ is $(k-2)$--coherent.
		\end{prop}
We postpone the proof of Proposition \ref{lem.limit boundary map in the same factor} to the next section.

\begin{prop}\label{prop.limit of k-par hyp}
	The limit representation $\r_0$ is $\{1,\ldots,k-3\}$--Anosov.
\end{prop}

\begin{proof}
	The representation $\r_0$ is $\{1,\ldots,k-2\}$--almost Anosov by Proposition \ref{lem.almost Anosov limit}; it remains to show Condition $(3)$ of our definition of Anosov representations, namely the growth of the eigenvalue gaps.
	Let $\r_0=\eta_1\oplus \eta_2$ be the{{ splitting coming from $(k-2)$--coherence.}}
	Then, as the almost Anosov boundary maps are dynamics preserving, we have that $\lambda^{\r_0}_j(\g)=\lambda^{\eta_1}_j(\g)$ for all $\g\in \G$ and all $j\leq k-2$. In particular it is enough to show that $\eta_1$ is $\{1,\ldots,k-3\}$--Anosov, which follows from Proposition \ref{lem.dynamics preserving for strongly irred}, as $\eta_1$ is strongly irreducible (Lemma \ref{lem.GW strong irreducibility}) and $\{1,\ldots,k-2\}$--almost Anosov. 
\end{proof}

Since $\r_n$ satisfies property $H_{k-2}^*$, Proposition \ref{prop.limit of k-par hyp} and Corollary \ref{cor.closdeness Hk under Anosov} imply 

\begin{prop}\label{cor.k-1 Anosov limit}
	If $\r_0$ is  $(k-1)$--Anosov, then $\r_0$ is $(k-2)$--Anosov.
\end{prop}
{{To deal with $k$--Anosov limits we need to ensure that the $k-2$ boundary maps converge to the almost Anosov boundary maps everywhere: this will put us in the position to apply once more Corollary \ref{cor.closdeness Hk under Anosov}}}.
\begin{lem}\label{lem.convergence of k-1 boundary map}
	If $\r_0$ is  $k$--Anosov, then for $j=k-2,d-k+2$ the boundary map $\xi^{j}_{\r_n}$ converges pointwise to the almost Anosov boundary map $\xi^{j}_{\r_0}$. 
\end{lem}

\begin{proof}
	Choose some non-trivial $\g\in\G$. We know that $\xi^j_{\r_n}(\g_+)\to \xi^j_{\r_0}(\g_+)$ for $j=k-3,k-2,k$. As always we conjugate $\rho_n$ so that for all $n\in\N$,
	$$\xi^j_{\r_n}(\g_+)=\xi^j_{\r_0}(\g_+),\quad j=k-3,k-2,k.$$
	We will denote the latter subspaces by $\g_+^j$, as they do not depend on $n$.
	
	Since $\r_n$ satisfies property $H_{k-2}$ and $C_{k-2}$ (Proposition \ref{prop.k-pos and Hk+Ck}), Proposition \ref{prop.Ck yields convex curve} gives that, for all $n\neq 0$,
	\begin{align*}
		\sfp_{n}&:\dg\backslash\{\g_+\}\to \P(\quotient{ \g_+^{d-k+3}}{ \g_+^{d-k}})\cong\R\P^2\\
		\sfp_n&(x):= [x^{k-2}_{\r_n}\cap \g_+^{d-k+3}]
	\end{align*}
	has  $C^1$--image and parametrizes the boundary of a strictly convex domain. Furthermore, since the boundary maps of $\rho_0$  are transverse, the map $\sfp_0(x):=[x^{k-2}_{\r_0}\cap \g_+^{d-k+3}]$ is well defined and continuous.  Since the boundary maps are dynamics preserving, the map $\sfp_n$ converges to $\sfp_0$ at fixed points in $\dg\backslash\{\g_+\}$. Since $\sfp_n$ parametrize the boundary of strictly convex domains and $\sfp_0$ is continuous, we have $\sfp_n(x)\to \sfp_0(x)$ for all $x\in \dg\backslash\{\g_+\}$. To see this first observe that $\sfp_0$ is a locally convex curve as it is the limit of convex curves on a dense set. For any $y\in \dg$ choose $x_1,x_2,x_3,x_4\in \dg$ such that $\sfp_n(x_i)\to \sfp_0(x_i)$ (e.g. fixed points of elements of $\G$) and $x_1,x_2,y,x_3,x_4$ are in that cyclic order. If we choose an affine chart around $y$ and $x_i$ close enough to $y$, strict convexity of $\sfp_n$ implies that $\sfp_n(y)$ is contained in the triangle with sides contained in the lines through $\sfp_n(x_1),\sfp_n(x_2)$ and $\sfp_n(x_2),\sfp_n(x_3)$ and $\sfp_n(x_3),\sfp_n(x_4)$, respectively. Hence any accumulation point of $\{\sfp_n(y)\}_{n\in \N}$ lies in the triangle with sides in $\sfp_0(x_1),\sfp_0(x_2)$ and $\sfp_0(x_2),\sfp_0(x_3)$ and $\sfp_0(x_3),\sfp_0(x_4)$, respectively. Choosing the $x_i$ closer and closer to $y$ reduces, by continuity and convexity of $\sfp_0$, this triangle to the point $\sfp_0(y)$.
	
	This is enough to conclude: for any accumulation point $V$ of $\{\xi^{k-2}_{\r_n}(x)\}_{n\in \N}$, we have $\xi^{k-3}_{\r_0}(x)\subset V\subset  \xi^{k}_{\r_0}(x)$ because $\r_0$ is $\{k-3,k\}$--Anosov; furthermore the projection $[V\cap \g_+^{d-k+3}]$ is a single point since $\xi^{k-3}_{\r_0}(x)\tv \xi^{d-k+3}_{\r_0}(\g_+)$ and $\xi^{k}_{\r_0}(x)\tv \xi^{d-k}_{\r_0}(\g_+)$, and this point uniquely determines $V$ together with $\xi^{k-3}_{\r_0}(x)$ and $\xi^{k}_{\r_0}(x)$. Since $\sfp_n(x)\to \sfp_0(x)$, the claim for $j=k-2$ follows. The case $j=d-k+2$ follows in the same way considering $\r_n^{\dual}$ converging to $\r_0^{\dual}$ instead.
\end{proof}

\begin{prop}\label{prop.k Anosov limt}
	If $\r_0$ is $k$--Anosov, then $\r_0$ is $\{1,\ldots,k\}$--Anosov.
\end{prop}

\begin{proof}
	Lemma \ref{lem.convergence of k-1 boundary map} and Proposition \ref{prop.k-pos and Hk+Ck} allow to apply Corollary \ref{cor.closdeness Hk under Anosov} to derive that $\r_0$ is $(k-1)$--Anosov. Thus Propositions \ref{prop.limit of k-par hyp} and \ref{cor.k-1 Anosov limit} yield that $\r_0$ is also $\{1,\ldots,k-2\}$--Anosov.
\end{proof}

\subsection{Coherence of the boundary map in the limit}\label{s.tecnical}% A technical transversality property of $k$--positive representations}
In this technical section we show that $k$--positive representations satisfy a transversality property similar to property $C_k$ which guarantees that certain projections to $\R\P^2$ parametrize the boundary of strictly convex domains (Corollary \ref{cor.strict convex for mixed quotient}). This is the key technical result needed in the proof that limits of $k$-positive representations are $(k-2)$-coherent (Proposition \ref{lem.limit boundary map in the same factor}).

\begin{center}
	\begin{figure}[h]	
		\begin{tikzpicture}
			\draw (1,0) circle [radius =1];
			\filldraw (0,0) circle [radius=1pt] node [left] {$x$};
			\filldraw (0.2,0.59) circle [radius=1pt] node [left] {$y_0$};
			\filldraw (0.6,0.9) circle [radius=1pt] node [above ] {$y$};
			\filldraw (1.4,0.9) circle [radius=1pt] node [above] {$z$};
			\filldraw (1.8,0.59) circle [radius=1pt] node [right] {$z_0$};
			\filldraw (2,0) circle [radius=1pt] node [right] {$w$};
			%\filldraw (1.4,-0.92) circle [radius=1pt] node [below] {$w_0$};
			\filldraw (0.6,-0.92) circle [radius=1pt] node [below] {$u$};
		\end{tikzpicture}
		\caption{The configuration of points in Lemma \ref{lem.strict convex for mixed quotient}}\label{figure:lem.strict convex for mixed quotient}
	\end{figure}
\end{center}
\begin{lem}\label{lem.strict convex for mixed quotient}
Let $\r$ be $k$--positive and $(x,y_0,z_0,w)\in \dg^{(4)}$ be cyclically oriented (see Figure \ref{figure:lem.strict convex for mixed quotient}). 
Let $\mathcal{U}$ be the set of $ u\in (x,w]_{y_0}$ such that the sum
$$(x^{d-k+1}\cap u^{d-1})+y^j + z^{k-j}$$
is direct for all $y\neq z\in [y_0,z_0]_x$ and $j=0,\ldots,k$.
Then $\mathcal{U}$ is open and contains a  set of the form $(x,t)_{y_0}$ for some $t\in (x,w]_{y_0}$.
\end{lem}

\begin{proof}
Since $y^j + z^{k-j}\to a^k$ if $y,z\to a$ (Proposition \ref{prop.continuity of sums}), the set
$$\mathcal{V}:=\{y^j + z^{k-j} \in \Gr_k(E) | y\neq z\in [y_0,z_0]_x, j=0,\ldots, k\}$$
is {closed} and thus compact. Since transversality is an open condition, we find, for any $(d-k)$-dimensional vector space $W$ transverse to every $V\in \mathcal{V}$, an open neighbourhood $\calW$ of $W$ in $\Gr_{d-k}(E)$ such that all points in $\calW$ are transverse to all points of $\mathcal{V}$. This proves openness.

Since $\rho$ is $k^*$--tridirect, Proposition \ref{prop.continuity of sums} applied to the contragradient representation gives $x^{d-k+1}\cap u^{d-1}\to x^{d-k}$ as $u$ tends to $x$. Since $x^{d-k}$ is transverse to all elements in $\mathcal{V}$ (Lemma \ref{5.9}), the above argument also shows that $\mathcal{U}$  contains the interesection of $(x,w]_{y_0}$ with a  neighbourhood of $x$.
\end{proof}
We denote by $\mathcal{U}_0\subset \mathcal{U}$ the maximal connected subset for which Lemma \ref{lem.strict convex for mixed quotient} holds.
\begin{prop}\label{p.5.17}
Let $\r$ be $k$--positive and $(x,y_0,z_0,w)\in \dg^{(4)}$ be in that order. Then for any $u\in \mathcal U_0\subset (x,w]_{y_0}$ the map 
$$\begin{array}{rl}
\displaystyle \pi^u&:[y_0,z_0]_x\cup \{x\}\to \calF\left(\quotient{E}{(x^{d-k+1}\cap u^{d-1})}\right)\\\\
&\left\{\begin{array}{l}
\displaystyle \pi^u(y):=([y^1],\ldots,[y^{k-1}]),\quad y\neq x\\
\displaystyle\pi^u(x):=([x^{d-k+1}],\ldots,[x^{d-1}])
\end{array}
\right.
\end{array}$$
{ is positive.}
\end{prop}

\begin{proof}

The transversality assumptions in the definition of the set $\mathcal{U}$ (and the transversality of the Anosov boundary map) imply that $\pi^u(x)\tv\pi^u(y)$ and $\pi^u(y)\tv\pi^u(z)$ for $y\neq z\in[y_0,z_0]_x$.
Choose a %parametrization $u_{\cdot}:(0,1]\to \mathcal{U}$ and a 
continuous map $g_{\cdot}:\mathcal U_0\to\PGL(E)$ such that $g_u(x^{d-k+1}\cap u^{d-1})=x^{d-k}$ for all $u$ (and $g_x=\rm{id}$). Since $g_u$ induces an isomorphism 
 $$\ov g_u:\calF\left(\quotient{E}{(x^{d-k+1}\cap u^{d-1})}\right)\to \calF\left(\quotient{E}{x^{d-k}}\right)$$ such that $\ov g_u\circ \pi^u=\pi_{[x]}\circ g_u\circ\xi$, and $\ov g_u$ preserves transversality,
$$\pi_{[x]}\circ g_u\xi(x)\tv \pi_{[x]}\circ g_u\xi(y) \quad \text{ and }\quad \pi_{[x]}\circ g_u\xi(y)\tv \pi_{[x]}\circ g_u\xi(z).$$
Where we extended $\pi_{[x]}$ to $\xi(x)$ using Lemma \ref{5.9}.

Since positivity is preserved along continuous transverse maps (Corollary \ref{cor.deform positive tuples}), $$\pi_{[x]}\circ g_u\circ\xi|_{[y_0,z_0]_x\cup \{x\}}$$ 
maps positively oriented tuples to positive tuples for all $u\in\mathcal U_0$. This implies that $\pi^u$ is positive.
\end{proof}

\begin{lem}
In the notation of Lemma \ref{lem.strict convex for mixed quotient}, $\mathcal{U}_0=(x,w]_{y_0}$.
\end{lem}
\begin{proof}
It is enough to show that $\mathcal{U}_0$ is closed. As in the proof of Proposition \ref{prop.direct reps closed among Anosov} we will show that the failure of transversality at a point $u_0\in\partial \mathcal U_0$ contradicts strong irreducibility of the dominating factor. Let $\{u_n\}_{n\in \N}$ be a sequence of points in $\mathcal{U}_0$ converging to $u_0$, and let $a\neq b\in [y_0,z_0]_x$ such that the sum
$$ (x^{d-k+1}\cap u_0^{d-1})+a^{l+1}+b^{k-l-1}$$
is not direct (as always we assume $l$ minimal for which such points exist, $l\geq 0$ since, as the representation is $k^*$ tridirect, $(x^{d-k+1}\cap u_0^{d-1})\pitchfork b^{k}$). 
 We choose a sequence $\{h_n\}\subset\PGL(E)$ converging to the identity such that 
$$W:= h_n\left((x^{d-k+1}\cap u_n^{d-1})+a^l+b^{k-l-2}\right)$$ does not depend on $n$ and  
$$h_n(x^{d-k+1}\cap u_n^{d-1})=(x^{d-k+1}\cap u_0^{d-1}).$$
Let
 $\xi_n^W:[y_0,z_0]_x\to \P(\quotient{E}{W})$ be the map given, with slight abuse of notation, by
$$\left\{\begin{array}{l}
\xi_n^W(y)\:= [h_n y^1],\quad y\neq a,b\\
\xi_n^W(a):= [h_n a^{l+1}],\\
 \xi_n^W(b):= [h_n b^{k-l-1}].
\end{array}\right.$$
For every $n\neq  0$ this map is injective and continuous: The projection of $\xi|_{[y_0,z_0]_x}$ to $\calF( \quotient{E}{x^{d-k+1}\cap u_n^{d-1}})$ is a positive curve by Proposition \ref{p.5.17}. In particular we can apply the fact that a quotient of a positive curve is again positive (Proposition \ref{prop.positive-quotient}, cfr. also Proposition \ref{prop.hyperconvex properties}) to get injectivity when considering the quotient by $(x^{d-k+1}\cap u_n^{d-1})+a^l+b^{k-l-2}$. This quotient map is clearly continuous at points different form $a,b$ and the continuity at $a,b$ follows from the fact that $a^l+y^1\to a^{l+1}$ for $y\to a$ (Proposition \ref{prop.continuity of sums}) and similarly for $b$. Now $h_n$ intertwines taking the quotient with $(x^{d-k+1}\cap u_n^{d-1})+a^l+b^{k-l-2}$ and with $(x^{d-k+1}\cap u_0^{d-1})+a^l+b^{k-l-2}$ and thus the same result follows for $\xi_n^W$.

The hyperconvexity of the boundary map (Proposition \ref{prop.hyperconvex properties}) implies that the image of  $\xi_n^W$ does not contain $[h_n x^{d-k+1}]\to [x^{d-k+1}]$ in its image, and $[x^{d-k+1}]\tv [a^{l+1}]=[ b^{k-l-1}]$. Thus   
$$y^1\in (x^{d-k+1}\cap u_0^{d-1})+a^{l+1}+b^{k-l-2}$$ 
for every $y\in (a,b)_x$. We can proceed now as in the last paragraph of the proof of Proposition \ref{prop.direct reps closed among Anosov} to produce a contradiction.
\end{proof}

\begin{cor}\label{cor.strict convex for mixed quotient}
Let $\r$ be $k$--positive.  For ordered $(x,y_0,z_0,w)\in \dg^{(4)}$ the curve
\begin{align*}
p^{(1)}:& (y_0,z_0)_x\to \P\left(\quotient{x^{d-k+3}}{(x^{d-k+1}\cap w^{d-1})}\right)\\
&p^{(1)}(y):= [y^{k-2}\cap x^{d-k+3}]
\end{align*}
has  $C^1$ image with tangents 
$p^{(2)}(y):= [y^{k-1}\cap x^{d-k+3}]$ and parametrizes the boundary of a strictly convex domain.
\end{cor}

\begin{proof} Since $\r$ satisfies property $H_{k-2}$ (Proposition \ref{prop.k-pos and Hk+Ck}), it follows with Proposition \ref{prop.1-hyperconvex for property Hk} that $p^{(1)}$ has $C^1$ image with tangents given by $p^{(2)}$. It remains to show that $p^{(1)}(y)$ is transverse to $p^{(2)}(z)$ for all $y\neq z\in (y_0,z_0)_x$ (cfr. \cite[Proposition 4.17]{Beyrer-Pozzetti}).
Since $\pi^w$ is positive (Proposition \ref{p.5.17}), we get from Proposition \ref{prop.transversalities of pos triples}(3) that
$$(\pi^w(x^{d-k+3})\cap \pi^w(y^{k-2}))+(\pi^w(x^{d-k+3})\cap \pi^w(z^{k-1}))$$
is direct. Observing that $p$ is the $\pi^w(x)^3$-truncation of $\pi^w$, this yields that  $p^{(1)}(y)+p^{(2)}(z)$ is direct.
\end{proof}

We now have all the ingredients to prove Proposition \ref{lem.limit boundary map in the same factor}: reductive limits of $k$-positive representations are coherent.

\begin{proof}[Proof of Proposition \ref{lem.limit boundary map in the same factor}]
	Let $E_1$ be irreducible with $\xi^1_{\r_0}(\dg)\subset \P(E_1)$, and consider the splitting $\r_0=\eta_1\oplus \eta_2:\G\to\PGL(E_1\oplus E_2)$.
	For the sake of contradiction, let $j+1\leq k-2$ be minimal such that for some (and hence, by Lemma  \ref{lem.splitting boundary map for reducible reps}, for all) $x\in\dg$, it holds $\xi^{j+1}_{\r_0}(x)\cap E_2\neq \{0\}$; {{equivalently $j$ is maximal with $E_2<\xi^{d-j}_{\rho_0}(x)$.}}
	
	Fix $\g\in\G$ and  consider the projective plane $\P\left(\quotient{V}{W}\right)$, where
	\begin{align*}
		V&:=%(\g_+)_{\r_0}^{d-j+1}=
		(\g_+)_{\r_n}^{d-j+1},\\
		W&:=%(\g_+^{d-j-1})_{\r_0}\cap (\g_-^{d-1})_{\r_0}=
		(\g_+^{d-j-1})_{\r_n}\cap (\g_-^{d-1})_{\r_n},
	\end{align*}
	here, as always, we assume, up to conjugating $\r_n$ by  a sequence $g_n\in\PGL(E)$ converging to the identity, that the two subspaces do not depend on $n$, and denote by $[\cdot]:V\to \quotient{V}{W}$ the natural projection.
	
	Let $(\g_+,y,z,\g_-)\in\dg^{(4)}$ be in that order on $\dg$. By Corollary \ref{cor.strict convex for mixed quotient}, the map
	$$p_n^{(1)}:(y,z)_{\g_+}\to \P\left(\quotient{V}{W}\right), \quad x\mapsto [x_{\r_n}^j\cap V]$$
	has $C^1$ image with tangents given by $p_n^{(2)}(x)=[x_{\r_n}^{j+1}\cap V]$ and parametrizes the boundary of a strictly convex domain.
	
	On the other hand, the map $p_0^{(1)}$  is  non-constant and has image in the  projective line $[E_1]$, while all lines $p_0^{(2)}$ pass through the point $[E_2]$ which is disjoint from the real projective line $[E_1]$: indeed the image $[E_2]$ is reduced to a single point since $E_2\subset \g_+^{d-j}\cap\g_-^{d-1}$ and $\g_+^{d-j-1}\cap E_2$ has codimension 1 in $E_2$. The projection $[E_1]$ then gives a supplementary real projective line. The image of $p_0^{(1)}$ is contained in $[E_1]$ since $x_{\r_0}^j<E_1$, and is not constant  because $\eta_1$ is irreducible and thus strongly irreducible, and thus Proposition \ref{prop.labourie-irreducible} applies. Furthermore $[E_2]\in p_0^{(2)}(x)$ since $x^{j+1}$ meets $E_2$ by assumption. 
	
	\begin{figure}[h]
		\begin{tikzpicture}
			\draw (-2,-1.3) -- (-2,1.3);
			\node at (-2.3,-1.1) [below] {$[E_1]$};
			\filldraw (0,0) circle [radius=1pt] node [above] {$[E_2]$};
			\filldraw (-1.95,0.8) circle [radius=1pt] node [right] {$p_n^{(1)}(h_1^+)$};
			\filldraw (-2.1,0) circle [radius=1pt] node [left] {$p_n^{(1)}(h_2^+)$};
			\filldraw (-1.9,-0.8) circle [radius=1pt] node [right] {$p_n^{(1)}(h_3^+)$};
			\draw (-2.2,0) -- (1,-0.2);
			\node at (1,-0.2) [right] {$p_n^{(2)}(h_2^+)$};
		\end{tikzpicture}
		\caption{The points $p_n(h_i^+)$ for $i=1,2,3$ can not be joined by a convex curve which is at $p_n(h_2^+)$ tangent to $p_n^{(2)}(h_2^+)$. 
			%This happens in a suitable affine chart of $\P \left(\quotient{V}{W}\right)$.
		}\label{fig.proof of Anosov limit for par hyp}
	\end{figure}
	The two pictures are contradictory since $p_n^{(1)}(h^+)$ and $p_n^{(2)}(h^+)$ converge, respectively, to  $p_0^{(1)}(h^+)$ and $p_0^{(2)}(h^+)$ at fixed points of hyperbolic elements (cfr. Figure \ref{fig.proof of Anosov limit for par hyp}):
	if we  choose $h_i\in\G$ such that $(h_1^+,h_2^+,h_3^+)$ are in that order inside $(y,z)_{\g_+}$ { and $ (\lim p_n^{(1)}(h_1^+),\lim p_n^{(1)}(h_2^+),\lim p_n^{(1)}(h_3^+))$ are in that order in $[E_1]$}, 
	strict convexity implies  that the tangent $p_n^{(2)}(h_2^+)$ converges to $[E_1]$, which doesn't contain $[E_2]$ - a contradiction. 
\end{proof}

\section{Examples and open questions}\label{s.ex}
In this section we collect several examples illustrating some of the difficulties we have to face and the questions that remain open.
\subsection{All transversality could get lost in the limit}\label{s.7.1}
It is natural to wonder if it is really necessary, for the proof of Theorem \ref{thmINTRO.limits of pos ratio general+proximal}, to assume some transversality in the limit, namely that condition (Tr) holds, or if this is always  the case. We provide here  an example of a sequence of representations $\rho_n:\Z\to\SL(4,\R)$ such that for each $n$, $\rho_n(\gamma)$ is $k$-proximal for all $k$, but $\lim_n\xi_n^1(\gamma^+)<\xi_n^3(\gamma^-)$.
For this aim we choose $\lambda>1$ and consider the matrices 
$$\rho_n(\gamma)=\bpm\lambda+\frac1{n^2}&0&0&0\\ \frac1n&\lambda&0&0\\0&0&\lambda^{-1}&\epsilon_n\\0&0&0&(\lambda+\frac1{n^2})^{-1}\epm.$$
If we set $\epsilon_n=\frac{-n}{\lambda(n^2\lambda+1)}$, then
$$\begin{array}{l}
\xi^1_n(\gamma_+)=[1:n:0:0]\\
\xi^3_n(\gamma_+)=\langle e_1,e_2,e_3\rangle\\
\xi^1_n(\gamma_-)=[0:0:n:1]\\
\xi^3_n(\gamma_-)=\langle e_2,e_3,e_4\rangle.
\end{array}$$
In this example the sequence $\rho_n$ converges to the representation $\rho_0$ with $\rho_0(\gamma)=\left(\begin{smallmatrix}\lambda\Id_2&0\\0&\lambda^{-1}\Id_2\end{smallmatrix}\right)$. Furthermore $\lim_n\xi_n^1(\gamma_+) =[e_2]\in \xi^3_n(\gamma_-)$ and $\lim_n\xi_n^1(\gamma_-) =[e_3]\in \xi^3_n(\gamma_+)$. As a result all transversality is lost in the limit.

We expect that with a similar idea one can produce examples of representations of free groups with the same pathology. This doesn't  give an example of a sequence of positively ratioed representations that doesn't satisfy condition (Tr), still provides evidence that all transversality could get lost in the limit if the limiting representation is not proximal. 

\begin{question}
Is there a converging sequence of positively ratioed representations $\r_n$ such that all transversality is lost in the limits of $\xi^1_{\r_n}$ and $\xi^{d-1}_{\r_n}$?
\end{question}

\subsection{The set of positively ratioed representations is not open}
We provide  examples of positively ratioed representations that can be approximated by non-positively ratioed representations, showing that being positively ratioed is, in general, not an open property.

In the notation of Example \ref{ex.Fuchsian I}, consider the representation $\r:=\tau_{\un d}\circ \r_{hyp}$ where $\un d=(d_1,d_2=d_1-1,d_3,\ldots,d_l)$ such that $d_3<d_1-1$. The irreducible factor $\r_1$ of $\r$ associated to $\R^{d_1}$ is the composition of a hyperbolization with an irreducible representation and thus a Hitchin representation. As $\xi^1_{\r}$ has image in $\P(\R^{d_1})$ and $\R^{d_2}\oplus \ldots \oplus \R^{d_l}\subset \xi^{d-1}_{\r}(x)$ for all $x\in\dg$ (cfr. Lemma \ref{lem.splitting boundary map for reducible reps}), it follows from Remark \ref{rem.cross ratio as projective} that
$$\cro_1(x^1_{\r},y^{d-1}_{\r},z^{d-1}_{\r},w^{1}_{\r})=\cro_1(x^1_{\r_1},y^{d_1-1}_{\r_1},z^{d_1-1}_{\r_1},w^{1}_{\r_1}).$$
Thus $\r$ is 1-positively ratioed, as $\r_1$ is 1--positively ratioed.

It was shown in \cite[Corollary 7.8]{PSW2} that there is a neighbourhood $U\subset \Hom(\G,\PGL(E))$ of $\r$ consisting of $1$--Anosov representations such that for every irreducible representation $\eta\in U$ the curve $\xi^1_{\eta}(\dg)$ is not Lipschitz. If we assume that $\G$ is the fundamental group of a closed surface with high enough genus, then Zariski dense representations are dense in $\Hom(\G,\PGL(E))$ \cite{Kim-Pansu}. In particular $\r$ is the limit of Zariski dense $1$--Anosov representations $\r_n$ where $\xi^1_{\r_n}(\dg)$ is not Lipschitz.
From the next proposition, of independent interest, it follows that such $\r_n$ are not 1--positively ratioed.

\begin{prop}\label{prop:pos>rect}
Let $\r:\G\to \PGL(E)$ be k--positively ratioed and reductive. Then the circle $\xi^k(\dg)\subset\Gr_k(E)$ is a Lipschitz submanifold.%\footnote{This does not mean that the map $\xi^k$ is a Lipschitz, but just that the image is a Lipschitz curve in $\Gr_k(E)$.}
\end{prop}
\begin{proof}
Applying an exterior power representation we can assume that $\r$ is $1$--positively ratioed (cfr. Section \ref{s.ext}). As $\r$ is reductive, we can split it as $\r=\eta_1\oplus \eta_2 :\G\to \PGL(E_1\oplus E_2)$ such that $\eta_1$ is irreducible and $\xi^1_{\r}(\dg)\subset \P(E_1)$ (Lemma \ref{lem.splitting boundary map for reducible reps}). In this case $\eta_1$ is $1$--Anosov, positively ratioed and $\xi^1_{\eta_1}=\xi^1_{\r}$. In particular we can assume without loss of generality that $\r$ is irreducible.

As Lipschitzness is a local property, it is enough to verify that, given an interval $(x,y)\subset \dg$ the image $\xi^1((x,y))$ is the graph of a Lipschitz function. Choose $d$ pairwise distinct points $x_1,\ldots,x_d\in\dg$ so that the associated linear subspaces $x_1^{d-1},\ldots, x_d^{d-1}$ are linearly independent (which we can by irreducibility of $\r$), and consider the induced projection
$$\Pi:\P(E)\setminus\bigcup_{i=1}^d x_i^{d-1}\to \prod_{i=1}^{d-1} \P\left(\quotient{E}{x_i\cap x_d}\right)
$$
It is easy to verify, by choosing suitable coordinates, that $\Pi$ is injective and smooth. Since the representation $\rho$ is Anosov, the image $\xi^1\left(\dg\setminus\{x_1,\ldots,x_d\}\right)$ is contained in  $\P(E)\setminus\bigcup_{i=1}^d x_i^{d-1}$, and it follows from positivity of the cross ratio that the restriction of $\Pi\circ\xi^1$ to any interval in its domain of definition is monotone on each coordinate (cfr. Remark \ref{rem.cross ratio as projective}). Thus its image is the graph of a Lipschitz function, as desired.
\end{proof}

We have shown in Corollary \ref{cor:Zd} that being positively ratioed is a closed condition among Zariski dense representations. If the condition were also open under this assumption, we would get that positively ratioed representations form connected components of  Zariski dense representations. We believe that this is  not the case.
\begin{question}
Is being positively ratioed an open condition among Zariski dense representations?
\end{question}

\subsection{Limits of $k$-positive representations}\label{s.ex k-pos degen} We discuss here the different cases of Theorem \ref{thm.intro-degenerations} in more detail.

\begin{example}[A reductive $(k-1)$--positive limit of $k$--positive representations]

In the notation of Example \ref{ex.Fuchsian I}, let $\r:=\tau_{(d_1,d_2)}\circ \r_{hyp}$ and  $k$ be maximal with $k<\frac{1}{2}(d_1-d_2)+1$. We assume that $k\geq 2$; $\r$ is $k$--positive (Example \ref{e.Fuchpos}). By definition $\r$ splits as $\eta_1\oplus \eta_2$ with $\eta_i:\G\to \PGL(\R^{d_i})$ in the Hitchin component of $\PGL(\R^{d_i})$. 
We find a continuous path of representations $t\to \r_t$ with $t\in [0,\infty)$ and $\r_0=\r$ such that for all $t$ the representation $\r_t=\eta_1^t\oplus \eta_2^t$ preserves the splitting $\R^{d_1}\oplus \R^{d_2}$ and we find lifts $\tilde{\eta}_1^t\oplus \tilde{\eta}_2^t\in \GL(\R^{d_1}\oplus \R^{d_2})$ of $\r_t$ such that $\tilde{\eta}_1^t$ is constant and $\lambda_1(\tilde{\eta}^t_2(\g))\to \infty$ for some $\g\in \G$. 
As soon as  for some $\g\in\G\backslash\{e\}$ we have $|\lambda_k(\tilde{\eta}_1(\g))|= |\lambda_1(\tilde{\eta}^t_2(\g))|$, $\r_t$ is not $k$--Anosov. In particular we find $t_0>0$ such that $\r_t$ is $k$--positive for $t<t_0$ and $\r_{t_0}$ is only $(k-1)$--positive.
\end{example}

It remains open if there also exist a limit of $k$--positive representations that is not $(k-1)$-Anosov or even not $(k-2)$-Anosov. 
To decide if there are examples of sequences of $k$--positive representations whose limit is not $(k-1)$-Anosov, it is enough to decide if there exists a sequence of representations satisfying $H_1^*$, namely that are 2--positive, whose limit is not 1--Anosov. 

\begin{question}
Is there a sequence of $H^*_1$ representations whose limit is not 1--Anosov?
\end{question}

To replace  $(k-3)$--positive with $(k-2)$--positive in Theorem \ref{thm.intro-degenerations}(1), it remains to control the growth of the $(k-2)$--th eigenvalue gap. We know that a limit $\r$ of $k$--positive representations splits as $\r=\eta_1\oplus \eta_2\to \PGL(E_1\oplus E_2)$, where $\eta_1$ is strongly irreducible, and there are continuous transverse boundary maps $\xi^j,\xi^{d-j}$ with  $\xi^j(\dg)\subset \Gr_j(E_1)$ for all $j=1,\ldots,k-2$. In light of Proposition \ref{lem.dynamics preserving for strongly irred} it would be enough to show that $|\lambda^{\eta_1}_{k-1}(\g)|\geq |\lambda^{\eta_2}_{1}(\g)|$ for all $\g\in \G$. For this one could try to argue as in the proof of Proposition \ref{lem.limit boundary map in the same factor}. However we don't know how to guarantee transversality of a limit $(\g_-)^{d-k+1}_{\r_0}:=\lim (\g_-)^{d-k+1}_{\r_n}$ and any $(h_+)^{k-1}_{\r_0}:=\lim (h_+)^{k-1}_{\r_n}$: In the notation of the proof of Proposition \ref{lem.limit boundary map in the same factor} for $j=k-2$, the missing transversality would mean that the projection of $(h_+)^{k-1}_{\r_0}$ is not defined in $\Gr_2(\quotient{V}{W})$ and therefore $p^2_n(h_+)\to [E_1]$ would not produce a contradiction to $(h_+)^{k-1}_{\r_0}\cap E_2\neq 0$. Thus we can refine the above question.

\begin{question}
Is there a converging sequence of $H^*_1$ representations $\r_n$ such that all transversality is lost in the limits of $\xi^1_{\r_n}$ and $\xi^{d-1}_{\r_n}$?
\end{question}

\bibliography{mybib}
\bibliographystyle{alpha}

\end{document}